\documentclass[a4paper,11pt]{amsart}
\usepackage{amsmath,  amsfonts, amssymb, amsthm, amscd}
\usepackage{graphicx}
\usepackage{enumerate}
\usepackage{float}
\usepackage{fontawesome5}
\usepackage{eso-pic}
\usepackage{charter}
\usepackage{url}
\usepackage{color}
\usepackage[utf8]{inputenc}
\usepackage{tikz}
\usetikzlibrary{arrows}
\usepackage[a4paper]{geometry}
\usepackage[colorlinks=true,linkcolor=blue,citecolor=magenta]{hyperref}
\usepackage{epsfig, enumerate}
\usepackage{graphicx}

\usepackage{booktabs}
\usepackage{dsfont}
\usepackage[utf8]{inputenc}
\usepackage[T1]{fontenc}
\usepackage{mathrsfs}
\usepackage{a4wide}
\usetikzlibrary{calc}
\usetikzlibrary{external}

\numberwithin{equation}{section}

\newtheorem{theorem}{Theorem}[section]
\newtheorem{lemma}[theorem]{Lemma}
\newtheorem{proposition}[theorem]{Proposition}

\newtheorem{remark}[theorem]{Remark}

\newcommand{\one}{\mathds{1}}

\newcommand{\mf}[1]{{\mathfrak #1}}

\newcommand{\bb}[1]{{\mathbb #1}}

\newcommand{\<}{\langle}
\renewcommand{\>}{\rangle}

\renewcommand{\epsilon}{\varepsilon}

\def\centerarc[#1](#2)(#3:#4:#5){\draw[#1] ($(#2)+({#5*cos(#3)},{#5*sin(#3)})$) arc (#3:#4:#5);}

\newcommand{\Lcal}{\mathcal{L}}

\DeclareMathSymbol{\leqslant}{\mathalpha}{AMSa}{"36} 
\DeclareMathSymbol{\geqslant}{\mathalpha}{AMSa}{"3E} 
\DeclareMathSymbol{\eset}{\mathalpha}{AMSb}{"3F}     

\newcommand{\R}{\mathbb{R}}

\newcommand{\Z}{\mathbb{Z}}

\renewcommand{\epsilon}{\varepsilon}

\newcommand{\E}{\mathbb{E}}

\newenvironment{myenumerate}{
	\renewcommand{\theenumi}{\arabic{enumi}}
	\renewcommand{\labelenumi}{{\rm(\theenumi)}}
	\begin{list}{\labelenumi}
		{
			\setlength{\itemsep}{0.4em}
			\setlength{\topsep}{0.5em}
			\setlength\leftmargin{2.45em}
			\setlength\labelwidth{2.05em}
			\setlength{\labelsep}{0.4em}
			\usecounter{enumi}
		}
	}
	{\end{list}
}

\newcommand{\beq}{\begin{equation}}
\newcommand{\eeq}{\end{equation}}
\newcommand{\ba}{\begin{aligned}}
	\newcommand{\ea}{\end{aligned}}

\usetikzlibrary{shapes.misc}
\usetikzlibrary{shapes.symbols}
\usetikzlibrary{shapes.geometric}
\usetikzlibrary{decorations}
\usetikzlibrary{decorations.markings}

\tikzstyle{tinydots}=[dash pattern=on \pgflinewidth off 2*\pgflinewidth]

\let\oldtocsection=\tocsection
\let\oldtocsubsection=\tocsubsection
\let\oldtocsubsubsection=\tocsubsubsection
\renewcommand{\tocsection}[2]{\hspace{0em}\oldtocsection{#1}{#2}}
\renewcommand{\tocsubsection}[2]{\hspace{1em}\oldtocsubsection{#1}{#2}}
\renewcommand{\tocsubsubsection}[2]{\hspace{2em}\oldtocsubsubsection{#1}{#2}}
\DeclareRobustCommand{\SkipTocEntry}[5]{}

\renewcommand{\P}{\mathbb{P}}

\title[Nonequilibrium fluctuations of the occupation time]{Nonequilibrium fluctuations for the occupation time of the SSEP in $d \geq 2$}

\author{Tiecheng Xu}
\address{Center for Mathematical Analysis, Geometry and Dynamical Systems, Instituto Superior T\'ecnico, Universidade de Lisboa, Av. Rovisco Pais, 1049-001 Lisboa, Portugal}
\email{xutcmath@gmail.com}

\author{Linjie Zhao}
\address{School of Mathematics and Statistics \& Hubei Key Laboratory of Engineering Modeling and Scientific Computing, Huazhong University of Science and Technology, Wuhan 430074, China.}
\email{linjie\_zhao@hust.edu.cn}

\begin{document}

\subjclass[2010]{60K35}

\begin{abstract}
We study the symmetric simple exclusion process in two or higher dimensions. We prove the invariance principles for the occupation time when the process starts from nonequilibrium measures. Our proof combines the martingale method and correlation estimates for the exclusion process.
	\end{abstract}

\keywords{Correlation estimates; exclusion process; nonequilibrium fluctuations; occupation time.}

\maketitle

\section{Introduction}

It has a long-standing history to study the occupation times of interacting particle systems. In the seminal paper \cite{kipnis1987fluctuations}, Kipnis investigated  equilibrium fluctuations for the occupation time of  the symmetric simple exclusion process (SSEP). More precisely, assume that the process starts from the Bernoulli product measure with constant density, which is reversible for the dynamics. Define the occupation time at the origin as 
\[\Gamma^n (t) := \beta_{d,n} \int_0^t \bar{\eta}_s (0) ds,\]
where $\{\eta_s, s \geq 0\}$ is the process accelerated by $n^2$ with $n$ being the scaling parameter, $\eta (0) \in \{0,1\}$ is the occupation number at the origin and $\bar{\eta} (0)$ is the centered occupation number, see Subsection \ref{sec: model} for rigorous definitions. Moreover,
\begin{equation}\label{beta d n}
\beta_{d,n} = \begin{cases}
		\sqrt{n}, \quad &\text{if} \quad d = 1,\\
		\frac{n}{\sqrt{\log n}}, \quad &\text{if} \quad d = 2,\\
		n, \quad &\text{if} \quad d \geq 3.\\
	\end{cases}
\end{equation}
Kipnis proved that for any $t > 0$, the occupation time $\Gamma^n (t)$ converges in distribution, as $n \rightarrow \infty$, to a normal distribution with explicit variance.  The CLT was extended to invariance principles by Sethuraman \cite{sethuraman2000central}.  The limit  turns out to be the fractional Brownian motion with Hurst parameter $3/4$ in $d = 1$ and to be the Brownian motion in $d \geq 2$. Kipnis and Varadhan in \cite{kipnis1986central} introduced the famous martingale method. Since then, a significant progress has been made to understand general additive functionals of particle systems, see Subsection \ref{subsec: literature} for a summary on the existing literature. 

We emphasize that the above literature focuses on the case when the initial measure is a stationary measure for the process. However, less is known when the initial measure is nonequilibrium.   
In one dimension, the nonequilibrium invariance principle for the occupation time of the SSEP was proved by the first author with Erhard and Franco in \cite{erhard2024nonequilibrium}.  The main motivation of this article is to extend the results in \cite{kipnis1987fluctuations,sethuraman2000central} to the nonequilibrium setting in higher dimensions. We prove that when the process starts from the Bernoulli product measure with a slowly varying profile, the  occupation time converges to a Gaussian process with covariance function explicitly given in $d \geq 2$. 
See the next subsection for  rigorous statements of our results.

\subsection{Main results}\label{sec: model}
The state space of the exclusion process is $\Omega = \{0,1\}^{\Z^d}$. For a configuration $\eta \in \Omega$ and a site $x \in \Z^d$, $\eta (x) \in \{0,1\}$ denotes the number of particles at site $x$.  The exclusion process is a continuous time Markov process with infinitesimal generator $\mathcal{L}$ acting on local functions $f: \Omega \rightarrow \R$ as 
\[\Lcal f (\eta) = \sum_{x \in \Z^d} \sum_{j=1}^d \Big(f(\eta^{x,x+e_j}) - f(\eta)\Big).\]
Above, we call $f$ a local function if its value depends on $\eta$ only through a finite number of coordinates. For $x,y \in \Z^d$, $\eta^{x,y}$ denotes the configuration obtained from $\eta$ by swapping the values of $\eta(x)$ and $\eta(y)$,
\[\eta^{x,y} (z) = \begin{cases}
	\eta(x), \quad &\text{if} \quad z = y,\\
	\eta(y), \quad &\text{if} \quad z = x,\\
	\eta(z), \quad &\text{otherwise}. 
\end{cases}\]

Let $\rho_0 : \R \rightarrow [0,1]$ be the initial density profile.  We assume that $\rho_0$ has a bounded fourth derivative. Let $n \in \mathbb{N} := \{1,2,\ldots\}$ be the scaling parameter. Define  $\nu^n_{\rho_0 (\cdot)}$ as the product Bernoulli measure on $\Omega$ with marginals given by
\[\nu^n_{\rho_0 (\cdot)} (\eta(x) = 1) = \rho_0 (x/n), \quad x \in \Z^d.\]
When $\rho_0 \equiv \rho \in [0,1]$ the constant profile, we simply write $\nu_\rho$. It is well known that $\nu_{\rho}$ is reversible for the symmetric simple exclusion process, see \cite{liggettips} for example.

We will speed up the process by $n^2$.  Denote by $\eta_t \equiv \eta^n_t$ the process with generator $\Lcal_n := n^2 \Lcal$. For any probability measure  $\mu$ on $\Omega$,  let $\P^n_\mu$ be the probability measure on the path space $D ([0,\infty); \Omega)$ induced by the process $\eta_t$ starting from the initial measure $\mu$, and let $\E_{\mu}^n$ be the corresponding expectation. 

We are interested in the  occupation time $\Gamma^n (t)$ at the origin, which is defined as
\[\Gamma^n (t) := \beta_{d,n} \int_0^t \bar{\eta}_s (0) ds,\]
where $\beta_{d,n}$ was defined in \eqref{beta d n} and for $x \in \Z^d$,
\[\bar{\eta}_s (x) = \eta_s (x) - \rho^n_s (x), \quad \rho^n_s (x) = \E^n_{\nu^n_{\rho_0 (\cdot)}} [\eta_s (x)], \quad x \in \Z^d.\]


  Let $\rho(t,\cdot)$ be the unique solution to the heat equation
  \begin{equation}\label{heat equation}
\begin{cases}
	\partial_t \rho (t,u) = \Delta \rho (t,u), &\quad t > 0, u \in \R^d,\\
	\rho(0,u) = \rho_0 (u), &\quad u \in \R^d.
\end{cases}
  \end{equation}

Throughout the article, we fix a time horizon $T > 0$. Below is the main result of this article.

\begin{theorem}\label{thm: fluctuation}
	As $n \rightarrow \infty$, the sequence of processes $\{\Gamma^n (t), 0 \leq t \leq T\}$ converges in distribution to some limit $\{\Gamma(t), 0 \leq t \leq T\}$ with respect to the measure $\P^n_{\nu^n_{\rho_0 (\cdot)}}$, in the space $C([0,T])$ endowed with the uniform topology, where 
	\[\Gamma (t) = \int_0^t \sigma_{d} (s) d B(s), \quad 0 \leq t \leq T.\]
	In this formula, $B(s)$ is the standard one dimensional Brownian motion and 
	\begin{equation}
		\sigma_{d}^2 (t) = \begin{cases}
			\chi(\rho(t,0)) /2 \pi, \quad &\text{if}\quad d=2;\\
			2 g_d (0)\chi(\rho(t,0)), \quad &\text{if}\quad d\geq 3,
		\end{cases} 
	\end{equation}
where $\chi (\rho) = \rho (1-\rho)$,  $g_d (0) = \int_0^t \mathfrak{q}_t (0,0) dt$ with $\mathfrak{q}_t$ being the transition probability of the continuous time random walk on $\Z^d$, which jumps to one of its neighbors at rate one.
\end{theorem}

\subsection{Related literature}\label{subsec: literature}  A more general problem is to study additive functionals of the exclusion process, which is defined  as $\int_0^t f(\eta_s) ds$, where $f: \Omega \rightarrow \R$ is a local function. Assume the exclusion process starts from the Bernoulli product measure with constant density $\rho \in (0,1)$. In \cite{kipnis1986central}, Kipnis and Varadhan showed that if the $H_{-1}$ norm of $f$ is finite (see \cite{komorowski2012fluctuations} for precise definitions of the $H_{-1}$ norm),  then the additive functional converges to the Brownian motion. In \cite{sethuraman1996central}, Sethuraman and Xu gave easily verified conditions under which the  $H_{-1}$ norm of $f$ is finite for reversible particle systems. The mean zero case was studied by Sethuraman in \cite{sethuraman2000central,sethuraman2006centralcor}.  In the asymmetric case, the behavior of the additive functionals depends on whether the density $\rho = 1/2$ or not.  When $\rho \neq 1/2$, invariance principles for the additive functionals were proved by Sepp{\"a}l{\"a}inen and Sethuraman \cite{seppalainen2003transience} in dimension one, and by Bernardin \cite{bernardin2004fluctuations} in dimension two. When $\rho = 1/2$, variance bounds for the occupation time at the origin were obtained by Bernardin \cite{bernardin2004fluctuations}, Sethuraman \cite{sethuraman2006superdiffusivity}, and Li and Mao \cite{li2008upper}. In \cite{gonccalves2013scaling}, Gon{\c c}alves and Jara proposed the local Boltzmann-Gibbs principle, which allows to prove invariance principles of additive functionals for particle systems in dimension one. Additive functionals of the exclusion process with long jumps were studied by Bernardin, Gon{\c c}alves and Sethuraman in  \cite{bernardin2016occupation}. See \cite{komorowski2012fluctuations} for an excellent review on this topic.

As we mentioned earlier, very few results concern nonequilibrium fluctuations of the additive functionals. In \cite{erhard2024nonequilibrium}, the first author with Erhard and Franco studied nonequilibrium fluctuations for the occupation time of the SSEP in dimension one by Fourier techniques and by calculating correlation estimates sharply, which allow them to relate the occupation time to the empirical measure of the process. In \cite{fontes2021additive}, the first author with Fontes investigated nonequilibrium fluctuations for the additive functionals of the weakly asymmetric simple exclusion in one dimension. Their proof is based on the sharp relative entropy bound by Jara and Menezes \cite{jara2018non}.  As far as we know, this paper is the first attempt to study nonequilibrium fluctuations of the occupation time in higher dimensions. Moreover, the covariance functions of the limiting Gaussian processes are explicitly given.

\subsection{Outline of the proof} Our proof is based on the martingale method introduced in \cite{kipnis1987fluctuations}, which is presented in  Section \ref{sec: proof}. Due to the self-duality of the SSEP, the resolvent equation for the occupation time can be solved explicitly. Then, we can decompose the occupation time as a martingale with a negligible term, see Subsection \ref{subsec: matingale decom}. However, since we are in the nonequilirbium setting, very little is known about the distribution of the process at time $t > 0$. Thus, it is not direct to obtain the convergence of the martingale term, which is proved in Subsection \ref{subsec: martingale} and needs the correlation estimates calculated in Section \ref{sec correlation}.   For the tightness of the negligible term, it also seems that we cannot use the Feynman-Kac technique in \cite{quastel2002central} directly due to the nonequilibrium setting. Instead, we prove the tightness of the occupation time directly in Section \ref{sec tightness} by using the correlation estimates from Section \ref{sec correlation}.

\subsection{Notation} Throughout this article, $C$ is a constant depending only on fixed parameters including $T$ and may change from line to line. We sometimes write $a\leq Cb$ simply as $a\lesssim b$.

\section{Proof of Theorem \ref{thm: fluctuation}}\label{sec: proof}

\subsection{Martingale decomposition}\label{subsec: matingale decom} Let $q_t = \mathfrak{q}_{tn^2}$ be the transition probability of the accelerated  random walk on $\Z^d$.  The dependence of $q_t$ on $n$ is omitted to make notation short. Then,
\begin{equation}\label{eq: pn}
	\partial_t q_t (0,x) = \Delta_n q_t (0,x), \quad q_t (0,x) = \delta_0 (x), \; x \in \Z^d,
\end{equation}
where $\delta_0$ is the Kronecker delta function on the origin and $\Delta_n$ is the discrete Laplacian, that is,  for $q: \Z^d \rightarrow \R$,
\[\Delta_n q (x) = n^2 \sum_{j=1}^d \big(q(x+e_j) + q(x-e_j) - 2q(x)\big).\]
By local central limit theorem(see Theorem 2.3.5 and Theorem 2.3.10 of \cite{Lawler2010} for instance), 
\begin{equation}\label{est 1}
	|q_t (0,x) - n^{-d} \bar{q}_t (0,\tfrac{x}{n})| \leq C \min \Big\{  \frac{1}{(tn^2)^{(d+2)/2}}, \frac{1}{(tn^2)^{d/2} |x|^2} \Big\}.
\end{equation}
where $\bar{q}_t$ is the Gaussian kernel,
\[\bar{q}_t (0,u) = \frac{1}{(4\pi t)^{d/2}} \exp \Big\{ - \frac{|u|^2}{4t} \Big\}, \quad |u|^2 := \sum_{j=1}^d u_j^2.\]
For any site $x \in \Z^d$ and any configuration $\eta \in \Omega$, define 
\begin{equation}
	g_n( x) = \int_0^\infty e^{-  t} q_t (0,x) dt, \quad G_n ( \eta) = \sum_{x \in \Z^d} g_n (x) \eta (x).
\end{equation}
Note that $G_n (\eta)$ is well defined since $\eta (x)$ is bounded in $x$ and $\sum_x g_n (x) = 1$. By \eqref{eq: pn} and integration by parts formula, 
\begin{equation}\label{eq: g_n property}
	(1 - \Delta_n) g_n (x) = \delta_0 (x), \quad x \in \Z^d.
\end{equation}

By Dynkin's martingale formula,
\[M_n (t):=  \beta_{d,n} G_n (\eta_t) - \beta_{d,n} G_n ( \eta_0) -   \beta_{d,n} \int_0^t  \mathcal{L}_n G_n ( \eta_s) ds \]
is a mean-zero martingale.
Since $\mathcal{L}_n \eta (x) = \Delta_n \eta (x)$, using the summation by parts formula and \eqref{eq: g_n property},
we have 
\[(1- \mathcal{L}_n) G_n ( \eta) = \eta (0).\]
Therefore, we have the following decomposition for the occupation time,
\begin{equation}\label{martingale decom}
	\Gamma^n (t) = M_n (t) + R_n (t),
\end{equation}
where
\begin{equation}
	R_n (t) := \beta_{d,n} G_n (\bar{\eta}_0) - \beta_{d,n} G_n (\bar{\eta}_t) +  \beta_{d,n} \int_0^t G_n (\bar{\eta}_s) ds.
\end{equation}

In the rest of this section, we deal with the two terms on the right hand side of \eqref{martingale decom} respectively.

\subsection{The martingale term}\label{subsec: martingale} In this subsection, we characterize the limit of the sequence of  martingales $\{M_n (t), 0 \leq t \leq T\}$. 

\begin{lemma}\label{lem mart converg}
	As $n \rightarrow \infty$, the sequence of martingales $\{M_n (t), 0 \leq t \leq T\}$ converges in distribution  to a martingale $\{M (t), 0 \leq t \leq T\}$, in the path space $D([0,T], \R)$ endowed with the Skorokhod topology, where 
	\[	d M (t) = \sigma_{d} (t) d B (t).\]
\end{lemma}

\begin{proof}   

We first prove the tightness of the martingale. 	 By Aldous' criterion, we only need to check the following two conditions:
\begin{itemize}
	\item[(1)]  for any $0 \leq t \leq T$,
	\[\lim_{M \rightarrow \infty} \limsup_{n \rightarrow \infty} \mathbb{P}^n_{\nu^n_{\rho_0 (\cdot)}} \big( \big| M_n (t)\big| > M \big) =0.\]
	\item[(2)] for any $\varepsilon > 0$, 
	\[\lim_{\gamma \rightarrow 0} \limsup_{n \rightarrow \infty} \sup_{\tau \in \mathcal{T}_T, \theta \leq \gamma} \mathbb{P}^n_{\nu^n_{\rho_0 (\cdot)}} \big( \big| M_n (\tau+\theta) - M_n (\tau)\big| > \varepsilon \big) =0,\]
	where $\mathcal{T}_T$ denotes the family of all  stopping times bounded by $T$.
\end{itemize}
Let $\< M_n\> (t)$ be the quadratic variation of the martingale $M_n (t)$.  	By direct calculations, 
\begin{align}
	\<M_n\> (t) &= \beta_{d,n}^2 \int_0^t ds \Big\{ \Lcal_n G_n (\eta_s)^2 - 2 G_n (\eta_s) \Lcal_n G_n (\eta_s) \Big\} \notag\\
	&=  \beta_{d,n}^2 \int_0^t ds \Big\{ n^2 \sum_{x \in \Z^d} \sum_{j=1}^d \big(g_n (x) - g_n (x+e_j)\big)^2 \big(\eta_s (x) - \eta_s (x+e_j)\big)^2  \Big\}.\label{mqv}
\end{align}
Multiplying $g_n (x)$ on both hands of \eqref{eq: g_n property}, summing over $x \in \Z^d$ and using the summation by parts formula, we have 
\[g_n (0) =  \sum_{x \in \Z^d} g_n (x)^2 + n^2 \sum_{x \in \Z^d} \sum_{j=1}^d \big(g_n (x) - g_n (x+e_j)\big)^2.\]
Since 
\[\sum_{x \in \Z^d} g_n (x)^2 = \int_0^\infty r e^{-r} q_r (0,0) dr, \]
by \eqref{est 1} and Tauberian's theorem \cite[Theorem XIII, 5.1]{feller1991introduction}, 
\[\sum_{x \in \Z^d} g_n (x)^2 \lesssim \begin{cases}
	n^{-1}  \quad &\text{if} \quad d =1,\\
	n^{-2} \quad &\text{if} \quad d =2,\\
	n^{-3} \quad &\text{if} \quad d =3,\\
	n^{-4} \log n \quad &\text{if} \quad d =4,\\
	n^{-4} \quad &\text{if} \quad d \geq 5.\\
\end{cases}\]
Thus, in dimensions $d \geq 2$,
\begin{equation}\label{gn square vanish}
	\lim_{n \rightarrow \infty} \beta_{d,n}^2 \sum_{x \in \Z^d} g_n (x)^2 = 0.
\end{equation}
This implies that
\begin{align}
	\lim_{n \rightarrow \infty} &n^2 \beta_{d,n}^2 \sum_{x \in \Z^d} \sum_{j=1}^d \big(g_n (x) - g_n (x+e_j)\big)^2 \notag\\
	&= \lim_{n \rightarrow \infty}  \beta_{d,n}^2 g_n ( 0)
	= \begin{cases}
		\frac{1}{4 \pi} \quad &\text{if}\quad d = 2,\\
		g_d (0) :=\int_0^\infty \mathfrak{q}_t (0,0) dt \quad &\text{if} \quad d \geq 3.
	\end{cases}\label{eq: g_n property 1}
\end{align}
where we used Tauberian's theorem in the last identity.  Since there is at most one particle at each site, in dimensions $d \geq 2$, 
\begin{equation}\label{eq: QV bound}
	\big| \<M_n\> (t) \big| \leq C t.
\end{equation}
Condition (1) follows immediately from the above estimate and Markov's inequality. To verify condition (2), using Markov's inequality and \eqref{eq: QV bound} again, we bound
\begin{align*}
	\mathbb{P}^n_{\nu^n_{\rho_0 (\cdot)}} &\big( \big| M_n (\tau+\theta) - M_n (\tau)\big| > \varepsilon \big) \\
	&\leq 	\varepsilon^{-2} \mathbb{E}^n_{\nu^n_{\rho_0 (\cdot)} }\Big[ \big(M_n (\tau+\theta) - M_n (\tau)\big)^2\Big] \\
		&= \varepsilon^{-2} \mathbb{E}^n_{\nu^n_{\rho_0 (\cdot)}} \Big[ \<M_n\> (\tau+\theta) - \<M_n\>(\tau)\Big] \leq C \theta.
\end{align*}
This proves the tightness of the martingale. 

Let $\{M(t)\}$ be the limit of $\{M_n (t)\}$ along some subsequence. Without loss of generality, let us still denote this subsequence by $\{n\}$. Next, we show the limit $\{M(t)\}$ has continuous trajectories. Indeed, since there is at most one particle can jump at each time, by  \eqref{eq: g_n property 1}, 
\begin{align*}
	\sup_{0 \leq t \leq T}  |M_n (t) - M_n (t-)| &= \sup_{0 \leq t \leq T} \beta_{d,n} |G_n (\eta_t) - G_n (\eta_{t-})| \\
	&\leq \sup_{x \in \Z^d, 1 \leq j \leq d} \beta_{d,n} |g_n ( x) - g_n (x+e_j)| \leq \frac{C}{n},
\end{align*}
which converges to zero as $n \rightarrow \infty$. 
	
Finally, to conclude the proof, by \cite[Theorem VIII, 3.11]{jacod2013limit}, it suffices to show that for any $t$,
\begin{equation}\label{qv congergence}
\lim_{n \rightarrow \infty} \<M_n\> (t) = \int_0^t \sigma_{d}^2 (s) ds
\end{equation}
in probability, which follows from the following two equations:
	\begin{multline}\label{eqn 1}
	\lim_{n \rightarrow \infty}  \E^n_{\nu^n_{\rho_0 (\cdot)}} \Big[ \Big| \<M_n\> (t) 
	-  \int_0^t  \Big\{ n^2 \beta_{d,n}^2 \sum_{x \in \Z^d} \sum_{j=1}^d \big(g_n (x) - g_n (x+e_j)\big)^2 \\
	\times \big( \rho^n_s (x) + \rho^n_s (x+e_j) - 2 \rho^n_s (x) \rho^n_s (x+e_j)\big) \Big\} ds \Big|\Big] = 0,
\end{multline}
\begin{multline}\label{convergence 1}
\lim_{n \rightarrow \infty} \int_0^t  \Big\{ n^2 \beta_{d,n}^2 \sum_{x \in \Z^d} \sum_{j=1}^d \big(g_n (x) - g_n (x+e_j)\big)^2 \\
\times
\big( \rho^n_s (x) + \rho^n_s (x+e_j) - 2 \rho^n_s (x) \rho^n_s (x+e_j)\big) \Big\} ds 
= \int_0^t \sigma_{d}^2 (s) ds.
\end{multline}

%

We first prove \eqref{eqn 1}.  By \eqref{eq: g_n property 1}, it suffices to show that, for any $1 \leq j \leq d$,
\[\lim_{n \rightarrow \infty} \sup_{x \in \Z^d} \E^n_{\nu^n_{\rho_0 (\cdot)}} \Big[ \Big|\int_0^t \bar{\eta}_s (x) ds \Big|\Big] = 0, \quad \lim_{n \rightarrow \infty} \sup_{x \in \Z^d} \E^n_{\nu^n_{\rho_0 (\cdot)}} \Big[ \Big|\int_0^t \bar{\eta}_s (x) \bar{\eta}_s (x+e_j) ds \Big|\Big] = 0.\]
By Cauchy-Schwarz inequality,
\begin{align*}
	\E^n_{\nu^n_{\rho_0 (\cdot)}} \Big[ \Big|\int_0^t \bar{\eta}_s (x) ds \Big|\Big]^2 \leq \E^n_{\nu^n_{\rho_0 (\cdot)}} \Big[ \Big|\int_0^t \bar{\eta}_s (x) ds \Big|^2 \Big] = \int_0^t \int_0^t  \E^n_{\nu^n_{\rho_0 (\cdot)}}  [\bar{\eta}_s (x) \bar{\eta}_r (x)] \,dr\,ds,
\end{align*}
and 
\begin{align*}
	\E^n_{\nu^n_{\rho_0 (\cdot)}}& \Big[ \Big|\int_0^t \bar{\eta}_s (x) \bar{\eta}_s (x+e_j)ds \Big|\Big]^2 \\
	&\leq \E^n_{\nu^n_{\rho_0 (\cdot)}} \Big[ \Big|\int_0^t \bar{\eta}_s (x) \bar{\eta}_s (x+e_j)ds \Big|^2 \Big] \\
	&= \int_0^t \int_0^t  \E^n_{\nu^n_{\rho_0 (\cdot)}}  [\bar{\eta}_s (x) \bar{\eta}_s (x+e_j)\bar{\eta}_r (x) \bar{\eta}_r (x+e_j)] \,dr\,ds.
\end{align*}
We then conclude the proof by using Lemma \ref{twotimes} and Lemma \ref{twotimestwopt} below. 

It remains to prove \eqref{convergence 1}. For any function $g: \Z^d \rightarrow \R$ and any $1 \leq j \leq d$, let us denote
\[\nabla_{n,j}^+ g (x) = n (g(x+e_j) - g(x)), \quad \nabla_{n,j}^- g (x) = n (g(x) - g(x-e_j)).\]
Since $\rho_0$ has a bounded fourth derivative, using Duhamel's representation and Taylor's expansion (see also Theorem A.1 of \cite{jara2006nonequilibrium}),
	\begin{equation}\label{eqn 2}
		\sup_{x \in \Z^d, 0 \leq s \leq T} \big|\rho^n_s (x) - \rho \big(s,\tfrac{x}{n}\big)\big| \lesssim n^{-2}.
	\end{equation}
Recall that $\rho(s,\cdot)$ is the solution to the heat equation \eqref{heat equation}. Together with \eqref{eq: g_n property 1}, we can replace $\big( \rho^n_s (x) + \rho^n_s (x+e_j) - 2 \rho^n_s (x) \rho^n_s (x+e_j)\big) $ by $2\chi (\rho(s,x/n))$ in \eqref{convergence 1}, and only need to deal with
\[2 \int_0^t \beta_{d,n}^2 \sum_{x \in \Z^d} \sum_{j=1}^d \big[\nabla_{n,j}^+ g_n (x)\big]^2 \chi (\rho(s,x/n)) ds.\]
Using the summation by parts formula, the last expression equals 
\begin{align*}
-2 \int_0^t \beta_{d,n}^2 \sum_{x \in \Z^d} \sum_{j=1}^d&  g_n (x) \nabla_{n,j}^- \big\{ \chi (\rho(s,x/n)) \nabla_{n,j}^+ g_n (x)\big\}ds\\
= -2 \int_0^t &\beta_{d,n}^2 \sum_{x \in \Z^d}   g_n (x) \chi(\rho(s,x/n)) \Delta_n g_n (x) ds  \\
&-2 \int_0^t \beta_{d,n}^2 \sum_{x \in \Z^d} \sum_{j=1}^d  g_n (x) \nabla_{n,j}^- \chi(\rho(s,x/n)) \nabla_{n,j}^+ g_n (x-e_j) ds.
\end{align*}
By \eqref{eq: g_n property}, the first term on the right hand side of the last equation equals 
\[2 \beta_{d,n}^2 g_n (0) \int_0^t \chi(\rho(s,0)) ds - 2 \int_0^t \beta_{d,n}^2 \sum_{x \in \Z^d} g_n (x)^2 \chi (\rho(s,x/n)) ds,\]
which converges to $\int_0^t \sigma_{d}^2 (s) ds$ by \eqref{gn square vanish},  \eqref{eq: g_n property 1} and the boundedness of $\chi(\rho(s,x/n))$. By Cauchy-Schwarz inequality, the second term on the right hand side is bounded by
\[2 \int_0^t \beta_{d,n}^2 \sqrt{\sum_{x \in \Z^d} \sum_{j=1}^d  g_n (x)^2 \big(\nabla_{n,j}^- \chi(\rho(s,x/n)) \big)^2} \sqrt{\sum_{x \in \Z^d} \sum_{j=1}^d \big(\nabla_{n,j}^+ g_n (x-e_j)\big)^2} ds,\]
which converges to zero by using \eqref{gn square vanish},  \eqref{eq: g_n property 1} again and the boundedness of $\nabla_{n,j}^- \chi(\rho(s,x/n))$.
This concludes the proof.
\end{proof}

\subsection{The term $R_n (t)$}\label{subsec: rnt}  In this subsection, we show that the term $R_n (t)$ vanishes in the limit for any $t$.

\begin{lemma}\label{lem: Rn}  For  any $t \geq 0$,
	\begin{equation}
		\lim_{n \rightarrow \infty} \E^n_{\nu^n_{\rho_0 (\cdot)}} \Big[ \big( R_n (t)\big)^2 \Big] = 0;
	\end{equation}
\end{lemma}

\begin{proof} We deal with the three terms in the definition of $R_n (t)$ respectively. By negative correlations of the symmetric exclusion process and since there is at most one particle at each site,
	\begin{equation}\label{Gn square bound}
		\E^n_{\nu^n_{\rho_0 (\cdot)}} \Big[ \big(G_n (\bar{\eta}_t)\big)^2 \Big] \leq \sum_{x \in \Z^d} g_n (x)^2 \E^n_{\nu^n_{\rho_0 (\cdot)}}  \big[ \bar{\eta_t} (x)^2\big] \leq \sum_{x \in \Z^d} g_n (x)^2.
	\end{equation}
	By \eqref{gn square vanish},  in dimensions $d \geq 2$, 
	\[ \lim_{n \rightarrow \infty} \E^n_{\nu^n_{\rho_0 (\cdot)}} \Big[ \big(\beta_{d,n} G_n (\bar{\eta}_t)\big)^2 \Big] =0.\]
	The term $G_n (\bar{\eta}_0)$ is easier since the initial measure is a product measure, and the estimate for the term $\int_0^t G_n (\bar{\eta}_s) ds$ follows from Cauchy-Schwarz inequality and \eqref{Gn square bound}. This concludes the proof.
\end{proof}

\subsection{Concluding the proof} In the last three sections, we have shown that the finite dimensional distribution of $\{\Gamma^n (t)\}$ converges to that of $\{\Gamma (t)\}$. In Section \ref{sec tightness}, we shall prove that the process $\{\Gamma^n (t)\}$ is tight in the space $C([0,T], \R)$ endowed with the uniform topology. This is enough to prove Theorem \ref{thm: fluctuation}.

\section{Correlation estimates}\label{sec correlation}

Given $k$ times $0\leq t_1\leq t_2 \leq \cdots \leq t_k\leq T$ and $k$ points $x_1,\dots, x_k$ in $\bb Z^d$, let us define the correlation function as
\begin{equation}\label{defgcor}
	\phi(t_1,\dots, t_k; x_1,\dots,x_k)\;:=\; \bb E^n_{\nu^n_{\rho_0(\cdot)}}\Big[\prod_{i=1}^k \overline{\eta}_{t_i}(x_i) \Big].
\end{equation}
Note that we do not require that $t_1,\ldots, t_k$ or that $x_1,\dots,x_k$ are distinct from each other. The aim of this section is to give an upper bound on the correlation function $\phi$. 

In \cite{erhard2024nonequilibrium}, a neat general upper bound for the correlation function with multiple times and arbitrary number of particles was obtained for the SSEP on $\bb Z$. However, for $d\geq 2$, such a general upper bound seems to be quite complex. Since our goals are just to prove the tightness of the occupation time and to verify \eqref{eqn 1}, below we will only estimate the space-time correlation terms needed for our purposes. Since we assume $T$ is fixed, to simplify the computation, we shall use the bound $\log(1+n^2t)\lesssim \log n$ frequently for all $t\leq T$.

\subsection{Correlation estimates at one time} The following lemma concerns the estimates of the correlation function at one time.

\begin{lemma}[one time]\label{onetime}
	Fix an integer $k\geq 2$.  Then there exists a constant $C=C(\rho_0,T)$ such that for all $0\leq t\leq T$,
	$$\sup_{x_1,\dots,\, x_k\,\,\text{\rm distinct}}\big\lvert \phi(t,\dots, t; x_1,\dots,x_k) \big\rvert \;\leq\; 
	\begin{cases}
	Cn^{-k}, \quad &\text{if}\;\;d\geq 3,\\
	Cn^{-k}\log^{k-1}(1+n^2 t), \quad &\text{if}\;\;d=2.
	\end{cases}$$
\end{lemma}

\begin{proof}
 To simplify the notation, we write
\begin{equation}\label{defvp}
	\varphi_t(x_i: 1\leq i\leq k)\;=\;  \bb E^n_{\nu^n_{\rho_0(\cdot)}}\Big[ \prod_{i=1}^{k} \overline{\eta}_{t}(x_i)\Big]\,,\qquad t\geq 0\,.
\end{equation}
Moreover, let us denote the vector $(x_i:1\leq i\leq k)$ by $\mathbf x$.
A tedious but straightforward computation (see also~\cite{ferrari1991symmetric}) shows that
\begin{equation*}
	\begin{split}
		\partial_t\varphi_t(\mathbf x)\;=\; &L_{k}\,\varphi_t(\mathbf x)\\
		&+ 2n^2 \sum_{i,j=1}^{k} \one\{|x_{j}-x_i|=1\}\big(\rho_t^n(x_{j})-\rho_t^n(x_i)\big)\big[\varphi_t(\mathbf x \backslash x_{j})-\varphi_t(\mathbf x \backslash x_i)\big]\\
		&- n^2 \sum_{i,j=1}^{k} \one\{|x_j-x_i|=1\}\big(\rho_t^n(x_{j})-\rho_t^n(x_i)\big)^2 \varphi_t(\mathbf x \backslash \{x_i,x_j\})\,,\\
	\end{split}
\end{equation*}
where $\varphi_t(\mathbf x\backslash A)=\varphi_t(\{x_i: 1\leq i\leq k\}\backslash A)$ for any subset $A\subset\{x_1,\dots,x_{k}\}$(when $A=\{x_i\}$ is a single point set, we simply write it as $x_i$), and $L_{k}$ is the generator of the SSEP with $k$ labelled particles accelerated by a factor $n^2$. Since labeling change does not affect the value of the correlation, strictly speaking, the actual dynamics we shall use in the correlation estimates is the stirring dynamics. Denote by $(\mathbf{X}^i:1\leq i\leq k)$ the accelerated stirring process with $k$ labelled particles, and by $\mathbf{P}_{(x_i\,:\,1\leq i\leq k)}$(resp.\ $\mathbf{E}_{(x_i\,:\,1\leq i\leq k)}$) the probability (resp.\ expectation) with respect to $(\mathbf{X}^i:1\leq i\leq k)$ starting from initial points $(x_i:1\leq i\leq k)$. By $\mathbf{X}^i_t$ we represent the location at time $t$ of the particle which was at site $x_i$ at time $0$.
Applying Duhamel's Principle, for any $t>0$, we can write $\varphi_t(\mathbf x)$ as the sum of an initial term and an integral term: for any $0\leq s <t$,
\begin{equation}\label{indeq}
		\begin{split}
			\varphi_t(\mathbf x)\;=\; \mathbf{E}_{\mathbf x}\big[\varphi_s(\mathbf{X}^i_{t-s}: 1\leq i\leq k) \big] \,+\mathbf{E}_{\mathbf x}\Big[ \int_s^t \Psi(\mathbf{X}_{t-r},r) dr\Big]
	\end{split}
\end{equation}
where for every integer $k\geq 1$,
\begin{equation}\label{defPsi}
	\begin{split}
		\Psi(\mathbf x,r) \;:=\;
		&2n^2  \sum_{i,j=1}^{k} \one\{|x_j-x_i|=1\}\big(\rho_{r}^n(x_j)-\rho_{r}^n(x_i)\big)\big[\varphi_{r}(\mathbf x\backslash x_j)-\varphi_{r}(\mathbf x\backslash x_i)\big]\\
		&- n^2\sum_{i,j=1}^{k} \one\{|x_j- x_i|=1\}\big(\rho_{r}^n(x_j)-\rho_{r}^n(x_i)\big)^2 \varphi_{r}(\mathbf x\backslash \{x_i,x_j\})\,.
	\end{split}
\end{equation}

Let
\begin{equation*}
	\Lambda_k \,=\,\big\{\mathbf w=(w_1,\cdots,w_j)\in(\bb Z^d)^k: w_i\neq w_j,\; \forall \; i\neq j  \}.
\end{equation*}
be the set of configurations with distinct points.
 The initial term in\eqref{indeq} can be written as  
\begin{equation*}
	\sum_{\mathbf y\in \Lambda_k}p_{t-s}(\mathbf x,\mathbf y)\,	\varphi_s(\mathbf y).
\end{equation*}
where $p_t(\cdot,\cdot)$ is the transition probability for the labeled stirring process speeded up by $n^2$. 

Define the maximal correlation at time $t$ by
\begin{equation*}
	A_t^k\,=\,\sup_{x_i's\; \text{are distinct}} 	|\varphi_t(x_i: 1\leq i\leq k)|.
\end{equation*}
Then obviously $A_t^1=0$ and we set $A_t^0=1$.

Using the trivial bound
\begin{equation}\label{difvarphi}
\big\lvert \varphi_{r}(\mathbf x\backslash x_j)-\varphi_{r}(\mathbf x\backslash x_i)\big\rvert \leq 2 A_r^{k-1},
\end{equation}
the integral term in \eqref{indeq} can be absolutely  bounded by
\begin{equation*}
	C\int_s^t 	\sum_{i,j=1}^{k} \mathbf{P}_{\mathbf x}[|\mathbf{X}_{t-r}^j- \mathbf{X}_{t-r}^i|=1]\big(n A_r^{k-1}\,+\,A_r^{k-2}\big)dr.
\end{equation*}
Combining these estimates with Lemma \ref{coupling} yields that
\begin{equation}\label{ind1}
	\begin{split}	
		|\varphi_t(\mathbf x)|\,
	\lesssim\,  \sum_{\mathbf y\in \Lambda_k}p_{t-s}(\mathbf x,\mathbf y)\,	\varphi_s(\mathbf y)\,+\,\int_s^t \Big(\frac{1}{1+n^2(t-r)}\Big)^{\frac{d}{2}} (nA_r^{k-1}+A_r^{k-2})dr,
	\end{split}
\end{equation}
for any $k\geq 2$.

Take $s=0$, then the initial term vanishes because the process starts from a Bernoulli product measure.
Thus
\begin{equation*}
	A_t^k\,\lesssim\,\int_0^t \Big(\frac{1}{1+n^2(t-r)}\Big)^{\frac{d}{2}} \Big(nA_r^{k-1}+A_r^{k-2}\Big)dr
\end{equation*} 
for $k\geq 2$.
This recursion relation, together with initial values $A_t^1=0$ and $A_t^0=1$ and Lemma \ref{intsingle},
 gives the desired result.
\end{proof}

It turns out that the correlation estimates established above are not sharp enough for our purpose  when $d=2$. The lack of sharpness originates from the loose inequality \eqref{difvarphi}.  To obtain a sharper correlation estimate in this case, we must refine our bound on the difference $\big\lvert \varphi_{r}(\mathbf x\backslash x_j)-\varphi_{r}(\mathbf x\backslash x_i)\big\rvert$.
\begin{lemma}\label{onetime2d}
Assume $d=2$. Fix an integer $k\geq 2$.  Then there exists a constant $C=C(\rho_0,T)$ such that for all $0\leq t\leq T$,
\begin{equation*}
\begin{split}
&\sup_{x_1,\dots,\, x_k\,\,\text{\rm distinct}}\big\lvert \phi(t,\dots, t; x_1,\dots,x_k) \big\rvert \\
&\leq\, 
\begin{cases}
	Cn^{-k}\log^{\frac{k}{2}}(1+n^2 t), \quad &\text{if}\;\;k\;\text{is even},\\
	Cn^{-k}\log^{\frac{k-1}{2}}(1+n^2 t)\log(1+\log(1+n^2 t)), \quad &\text{if}\;\;k\;\text{is odd}.
\end{cases}
\end{split}
\end{equation*}
\end{lemma}

Before proving the lemma, we introduce a coupling to be used in the proof. The coupling essentially follows the ideas of \cite{AIHPB_2008__44_2_341_0}. Fix $k\geq 2$, $\mathbf z\in\Lambda_{k-1}$, and two points $z_k,z_{k+1}\in\bb Z^d$ such that $|z_k-z_{k+1}|=1$ and $z_i$'s with $1\leq i\leq k+1$ are all distinct. Consider $k+1$ labeled particles evolving on $\bb Z^d$ according to the following rules. They start from $\mathbf z$, $z_k$, $z_{k+1}$ and evolve according to an (accelerated) stirring process. More precisely,  each (non-oriented) edge $(a,b)$ is associated to an independent Poisson clock with rate $n^2$. When the clock at edge $(a,b)$ rings, the contents at two ends swap. However, when the particles starting at $z_k$ and $z_{k+1}$ are at distance $1$, each one jumps, independently
from the other, to the site occupied by the other at the rate $n^2$. In other words, if these two particles at two ends of some edge $(a,b)$, then two independent Poisson clocks with rates $n^2$ are associated to edge $(a,b)$ of two orientations, and particles jump according to the clock without respecting the exclusion rule.
 Once these particles occupy the same site, they remain together forever. Notice that the two particles starting from $z_k$ and $z_{k+1}$ behave until they meet exactly as two independent particles. let $\tau$ be the first time that these two particles meet. Then in dimension two, 
 \begin{equation}\label{meeting}
 	\mathbf P_{(\mathbf z,z_k,z_{k+1})}\big[\tau>t\big]\,\lesssim\, \frac{1}{1+\log(1+n^2t)}.
 \end{equation}
 where, abusing the notation, $\mathbf P_{(\mathbf z,z_k,z_{k+1})}$(resp. $\mathbf E_{(\mathbf z,z_k,z_{k+1})}$) denotes the probability (resp. expectation) corresponding to the evolution just described. 
 
 Denote by $\mathbf Z\in\Lambda_{k+1}$ the vector of positions of particles starting from $\mathbf z, z_k,z_{k+1}$.  Notice that particles starting from $\mathbf z$ and either one from $z_k$ or $z_{k+1}$ form a stirring process with $k$ labeled particles. 

\begin{proof}[Proof of Lemma \ref{onetime2d}]
	Let us denote
	\begin{equation*}
		\mf A_t^{k-1}\,=\,\sup_{\substack{\mathbf x\in \Lambda_{k}\\1\leq i,j\leq k}}	\Big|\one\{|x_j-x_i|=1\}\big[\varphi_{t}(\mathbf x\backslash x_j)-\varphi_{t}(\mathbf x\backslash x_i)\big]\Big|
	\end{equation*}
	for every $k\geq 1$.
With this notation, we can estimate the absolute value of the integral term in \eqref{indeq} by 
\begin{equation*}
\int_s^t \sum_{i,j=1}^{k} \mathbf{P}_{(x_i\,:\, 1\leq i\leq k)}[|\mathbf{X}_{t-r}^j- \mathbf{X}_{t-r}^i|=1] (n\mf A_r^{k-1}+A_r^{k-2})dr.
\end{equation*}
This estimate together with the arguments from the proof of Theorem \ref{onetime} gives
\begin{equation}\label{2donetime1}
	A_t^k\,\lesssim\, \sup_{\mathbf x\in \Lambda_k}\sum_{\mathbf y\in \Lambda_k}p_{t-s}(\mathbf x,\mathbf y)\,	\varphi_s(\mathbf y)\,+\,\int_s^t \frac{1}{1+n^2(t-r)} (n\mf A_r^{k-1}+A_r^{k-2})dr.
\end{equation}
By \eqref{indeq} and \eqref{defPsi}, for $\mathbf x\in\Lambda_k$ such that $|x_i-x_j|=1$ for some $i<j$, let 
\begin{equation*}
	\mathbf z\,=\, \mathbf{x}\backslash\{x_i,x_j\}\in\Lambda_{k-2}, \quad z_k=x_i,\quad z_{k+1}=x_j.
\end{equation*}
Note that 
\begin{equation*}
	\mathbf x\backslash x_j\,=\,(x_1,\cdots,x_{j-1},x_{j+1},\cdots,x_k)
\end{equation*}
and 
\begin{equation*}
	(\mathbf z, z_k)\,=\,(x_1,\cdots,x_{i-1},x_{i+1},\cdots,x_{j-1},x_{j+1},\cdots,x_k,x_i)
\end{equation*}
are not necessarily equal. 
Since label changing does not affect the value of the correlation,
\begin{equation*}
\varphi_{r}(\mathbf x\backslash x_j)\,=\,\varphi_{r}((\mathbf z,z_k)),\quad \varphi_{r}(\mathbf x\backslash x_i)\,=\,\varphi_{r}((\mathbf z,z_{k+1})).
\end{equation*}
Therefore we have 
\begin{equation*}
	\begin{split}
	&\big\lvert \varphi_{t}(\mathbf x\backslash x_j)-\varphi_{t}(\mathbf x\backslash x_i)\big\rvert\\
	\lesssim\,&	\sum_{\mathbf y\in\Lambda_{k} }\big\lvert p_{t-s}((\mathbf z,z_k),\mathbf y)\,-\,p_{t-s}((\mathbf z,z_{k+1}),\mathbf y) \big\rvert\,	|\varphi_s(\mathbf y)|\\
	+\,&\Big\lvert\mathbf{E}_{(\mathbf z,z_k,z_{k+1})}\Big[ \int_s^t \Psi(\mathbf{Z}_{t-r}\backslash Z^k_{t-r},r) \,-\,\Psi(\mathbf{Z}_{t-r}\backslash Z^{k+1}_{t-r},r) dr\Big]\Big\rvert.
	\end{split}
\end{equation*}
Since $\mathbf{Z}\backslash Z^{k}$ and $\mathbf{Z}\backslash Z^{k+1}$ coincide  after the meeting time $\tau$, the second term can be written and estimated as 
\begin{equation*}
	\begin{split}
	&\Big\lvert\mathbf{E}_{(\mathbf z,z_k,z_{k+1})}\Big[ \int_s^t \one_{\tau>t-r}\big[ \Psi(\mathbf{Z}_{t-r}\backslash Z^k_{t-r},r) \,-\,\Psi(\mathbf{Z}_{t-r}\backslash Z^{k+1}_{t-r},r)\big] dr\Big]\Big\rvert\\
	\lesssim\,&\int_s^t \big\{ n\mf A_r^{k-2}+A_r^{k-3}\big\} \sum_{\substack{1\leq i<j\leq k+1\\ (i,j)\neq (k,k+1)}}\mathbf P_{(\mathbf z,z_k,z_{k+1})} \big[\tau>t-r, |\mathbf Z^i_{t-r}-\mathbf Z^j_{t-r}|=1\big]dr
	\end{split}
\end{equation*}
Replacing the indicator function $\one_{\tau>t-r}$ by $\one_{\tau>(t-r)/2}$,  then applying the Markov property at time $(t-r)/2$, and finally applying Lemma \ref{coupling} to be proved later, the previous expression can be bounded by
\begin{equation*}
	\begin{split}
&C\int_s^t \big\{ n\mf A_r^{k-2}+A_r^{k-3}\big\} \frac{1}{1+n^2(t-r)} \mathbf P_{(\mathbf z, \,z_k,z_{k+1})} \big[\tau>(t-r)/2 \big]dr\\
\lesssim\,&\int_s^t \big\{ n\mf A_r^{k-2}+A_r^{k-3}\big\} \frac{1}{1+n^2(t-r)}\frac{1}{1+\log(1+n^2(t-r))}  dr,
	\end{split}
\end{equation*}
where the last inequality is due to \eqref{meeting}.
From the above estimate we can conclude that,
\begin{equation}\label{2donetime2}
		\begin{split}
			\mf A_t^{k-1}\,\lesssim\, &  \sup_{\mathbf x\in\Lambda_{k}}\sum_{\mathbf y\in\Lambda_{k} }\big\lvert p_{t-s}((\mathbf z,z_k),\mathbf y)\,-\,p_{t-s}((\mathbf z,z_{k+1}),\mathbf y) \big\rvert\,	|\varphi_s(\mathbf y)|\\
			+\,& \int_0^t\frac{1}{1+n^2(t-r)} \frac{1}{1+\log(1+n^2(t-r))} (n\mf A_r^{k-2}+A_r^{k-3})dr.
		\end{split}
\end{equation}
Recursion relations \eqref{2donetime1} and \eqref{2donetime2} with $s=0$ give that, for $k\geq 2$, 
\begin{equation*}
	A_t^k\;\leq\; 
	\begin{cases}
		n^{-k}\log^{\frac{k}{2}}(1+n^2 t), \quad &\text{if}\;\;k\;\text{is even},\\
		n^{-k}\log^{\frac{k-1}{2}}(1+n^2 t)\log(1+\log(1+n^2 t)), \quad &\text{if}\;\;k\;\text{is odd},
	\end{cases}
\end{equation*}
and 
\begin{equation*}
\mf	A_t^k\;\leq\; 
	\begin{cases}
		n^{-k}\log^{\frac{k-2}{2}}(1+n^2 t)\log(1+\log(1+n^2 t)), \quad &\text{if}\;\;k\;\text{is even},\\
		n^{-k}\log^{\frac{k-3}{2}}(1+n^2 t)\log^2(1+\log(1+n^2 t)), \quad &\text{if}\;\;k\;\text{is odd}.
	\end{cases}
\end{equation*}
In fact, to verify these bounds, one just needs to compute the bounds for $k=2,3$ and then observes that the bound for $A_t^k$ is increasing in time, the bound for $n\mf A_r^{k-1}$ is less than that of $A_r^{k-2}$, using \eqref{eqd2} and 
\begin{equation*}
	\int_s^t \frac{1}{1+n^2(t-r)} \frac{1}{1+\log(1+n^2(t-r))} dr\,\leq\,\frac{1}{n^2} \log(1+\log(1+n^2 (t-s))),
\end{equation*}
which would give 
\begin{equation*}
	A_t^k\,\lesssim\,\frac{1}{n^2}\log(1+n^2 t)A_t^{k-2}
\end{equation*}
and 
\begin{equation*}
	\mf A_t^k\,\lesssim\,\frac{1}{n^2}\log \big( 1+ \log(1+n^2 t) \big) A_t^{k-2}.
\end{equation*}
\end{proof}

We can extend the previous results to allow repetitive points, namely for points $\mathbf x=(x_1,\cdots,x_k)$ with $x_i=x_j$ for some $i,j$. 
For instance, using  the  identity
\begin{equation*}
	\overline{\eta}(x)^2\;=\;(1-2\rho^n(x))\overline{\eta}(x)+ \rho^n(x)-\big(\rho^n(x)\big)^2\,,
\end{equation*}
one have for a list of non-repetitive points $(x_i: 1\leq i\leq k)$,
\begin{equation}\label{reduce}
	\sup_{x_1,\dots,\, x_k\,\,\text{\rm distinct}}\big\lvert \phi(t,\dots, t,t; x_1,\dots,x_k,x_k) \big\rvert\,\lesssim\, A_t^k+A_t^{k-1},
\end{equation}
and
\begin{equation}\label{reduce1}
	\sup_{x_1,\dots,\, x_k\,\,\text{\rm distinct}}\big\lvert \phi(t,\dots, t,t; x_1,\dots,x_{k-1},x_k,x_{k-1},x_k) \big\rvert\,\lesssim\, A_t^k+A_t^{k-1}+A_t^{k-2},
\end{equation}
and so on.


\subsection{Correlation estimates at two times} In this subsection, we estimate the correlation function at two times.
\begin{lemma}[two times]\label{twotimes}
Fix $0< s<t\leq T$ and $y\in\bb Z^d$. Assume $x_1,x_2,x_3$ are distinct. Then 
\begin{equation*}
	\begin{split}
		|\phi(s,t;y,x_1)|\,	\lesssim\,
		\begin{cases}
			\frac{1}{n^2}\log n\,+\,\frac{1}{1+n^2(t-s)} \quad &\text{if}\;\; d=2\\
			\frac{1}{n^2}\,+\,\Big(\frac{1}{1+n^2(t-s)}\Big)^{\frac{d}{2}} \quad &\text{if}\;\; d\geq 3,
		\end{cases}
	\end{split}
\end{equation*}
and
\begin{equation*}
	\begin{split}
	|\phi(s,t,t;y,x_1,x_2)|\,\lesssim\,
		\begin{cases}
		\frac{\log n\log\log n}{n^3}\,+\,\frac{1}{n(1+n^2(t-s))} \quad &\text{if}\;\; d=2\\
			\frac{1}{n^3}\,+\,\frac{1}{n}\Big(\frac{1}{1+n^2(t-s)}\Big)^{\frac{d}{2}} \quad &\text{if}\;\; d\geq 3,
		\end{cases}
	\end{split}
\end{equation*}
and
\begin{equation*}
|\phi(s,t,t,t;y,x_1,x_2,x_3)|
		\lesssim
		\begin{cases}
			\frac{\log^2n}{n^4}\,+\,\frac{\log n}{n^2(1+n^2(t-s))}\quad &\text{if}\;\; d=2\\
			\frac{1}{n^4}\,+\,\frac{1}{n^2}\Big(\frac{1}{1+n^2(t-s)}\Big)^{\frac{d}{2}} \quad &\text{if}\;\; d\geq 3.
		\end{cases}
\end{equation*}
\end{lemma}

\begin{proof}
	Let us denote 
	\begin{equation*}
		B_t^k\,=\,\sup_{x_i's\; \text{are distinct}} 	|\phi(s,t,\cdots,t;y,x_1,\cdots,x_k)|.
	\end{equation*}
	Then obviously $B_t^0=A_s^1=0$. 
	
	We first deal with the case $d\geq 3$.  Following a similar procedure as in the derivation of \eqref{ind1}, we obtain the estimate:
	\begin{equation}\label{ind2}
		\begin{split}
			&|\phi(s,t,\cdots,t;y,x_1,\cdots,x_k)|\\
			\lesssim\, & \sum_{(y_i:1\leq i\leq k)\in\Lambda_{k}}p_{t-s}((x_i:1\leq i\leq k),(y_i:1\leq i\leq k))\,	|\phi(s,s,\cdots,s;y,y_1,\cdots,y_k)|\\ +\,&\one_{\{k\geq 2\}}\int_s^t\Big(\frac{1}{1+n^2(t-r)}\Big)^{\frac{d}{2}} (nB_r^{k-1}+B_r^{k-2})dr,
		\end{split}
	\end{equation}
	for all $k\geq 1$. The initial term is estimated by considering whether some $y_i$  is equal to $y$ or not. If $y_i\neq y$ for any $i$, then the total contribution is bounded simply by $A_s^{k+1}$. If there exists some $i$ such that $y_i=y$ (which can occur for at most one index), then by Lemma \ref{fixy} below and \eqref{reduce}, the total contribution is bounded by
	\begin{equation*}
		C\Big(\frac{1}{n^2(t-s)+1}\Big)^{d/2} (A_s^k+A_s^{k-1}).
	\end{equation*}
Combining these estimates together, we have
	\begin{equation}\label{indtwo}
	\begin{split}
		B_t^k\lesssim &\, A_s^{k+1}\,+\,\Big(\frac{1}{n^2(t-s)+1}\Big)^{\frac{d}{2}} (A_s^k+A_s^{k-1})\\ +\,&\one_{\{k\geq 2\}}\int_s^t\Big(\frac{1}{1+n^2(t-r)}\Big)^{\frac{d}{2}} (nB_r^{k-1}+B_r^{k-2})dr.
	\end{split}
	\end{equation}
	We now use this formula to estimate $B_t^k$ for $1\leq k\leq 3$.
	
For $k=1$, Theorem \ref{onetime} gives 
\begin{equation*}
	B_t^1\,\lesssim\,\frac{1}{n^2}\,+\,\Big(\frac{1}{1+n^2(t-s)}\Big)^{\frac{d}{2}}.
\end{equation*}

For $k=2$,  the first line at the right hand side of \eqref{indtwo} is bounded by
\begin{equation*}
	\frac{C}{n^3}\,+\,\Big(\frac{1}{1+n^2(t-s)}\Big)^{\frac{d}{2}}\frac{C}{n^2}.
\end{equation*}
 The integral term is bounded by 
\begin{equation*}
 C\int_s^t \Big(\frac{1}{1+n^2(t-r)}\Big)^{\frac{d}{2}}  nB_r^1dr \lesssim	\frac{1}{n^3}\,+\,\frac{1}{n}\Big(\frac{1}{(1+n^2(t-s))}\Big)^{\frac{d}{2}}
\end{equation*}
 by \eqref{eqd3}. These estimates above give the desired bound.

For $k=3$, the sum in the first line at the right hand side of \eqref{indtwo} is bounded  by
\begin{equation*}
	\frac{C}{n^4}\,+\,\frac{C}{n^2}\Big(\frac{1}{1+n^2(t-s)}\Big)^{\frac{d}{2}}.
\end{equation*}
 The integral term is bounded by 
\begin{equation*}
	C \int_s^t \Big(\frac{1}{1+n^2(t-r)}\Big)^{\frac{d}{2}} (nB_r^2+B_r^1)dr\,\lesssim\,\frac{1}{n^4}\,+\,\frac{1}{n^2}\Big(\frac{1}{1+n^2(t-s)}\Big)^{\frac{d}{2}}
\end{equation*}
by \eqref{eqd3} and \eqref{eq3d3}. 
 
 We now consider the case $d=2$. Define 
 \begin{equation*}
 	\mf B_t^{k-1}\,=\,\sup_{\substack{\mathbf x\in \Lambda_{k}\\1\leq i,j\leq k}}	\Big|\one\{|x_j-x_i|=1\}\big[\phi(s,t,\cdots, t; y,\mathbf x\backslash x_j)-\phi(s,t,\cdots,t;y,\mathbf x\backslash x_i)\big]\Big|
 \end{equation*}
 for every $k\geq 1$. Using arguments similar to those leading to \eqref{2donetime1} and \eqref{indtwo}, we obtain
 \begin{equation}\label{Bt}
 	\begin{split}
 	B_t^k\,\lesssim\, &A_s^{k+1}\,+\,\frac{1}{n^2(t-s)+1} (A_s^k+A_s^{k-1})\\
 	+\,&\one_{\{k\geq 2\}}\int_s^t \frac{1}{1+n^2(t-r)} (n\mf B_r^{k-1}+B_r^{k-2})dr.
 	\end{split}
 \end{equation}
 We can apply  similar argument for \eqref{2donetime2} to get that, 
  for $\mathbf x\in\Lambda_k$ such that $|x_i-x_j|=1$ for some $i$ and $j$, 
 \begin{equation*}
 	\begin{split}
 	&\big\lvert\phi(s,t,\cdots, t; y,\mathbf x\backslash x_j)-\phi(s,t,\cdots,t;y,\mathbf x\backslash x_i)\big\rvert\\
 	\lesssim\, &  \sum_{\mathbf y\in\Lambda_{k-1} }\big\lvert p_{t-s}((\mathbf{x}\backslash\{x_i,x_j\},x_i),\mathbf y)\,-\,p_{t-s}((\mathbf{x}\backslash\{x_i,x_j\},x_j),\mathbf y) \big\rvert\,	|\phi(s,s,\cdots,s;y,\mathbf y)|\\
 		+\,& \int_s^t\frac{1}{1+n^2(t-r)} \frac{1}{1+\log(1+n^2(t-r))} (n\mf B_r^{k-2}+B_r^{k-3})dr.
 	\end{split}
 \end{equation*}
 Using \eqref{reduce}, Lemmas \ref{difp} and  Lemma \ref{difpfixz}  below with $|I|=1$, we get
 \begin{equation}\label{mfBt}
 	\begin{split}
 	\mf B_t^k\,\lesssim\,&   \frac{\log(1+n^2(t-s))}{\sqrt{1+n^2(t-s)}}	A_s^{k+1}\,+\, \Big(\frac{1}{1+n^2(t-s)}\Big)^{\frac{3}{2}}\big[A_s^k+A_s^{k-1}\big]\\
 	+\,& \one_{\{k\geq 2\}}\int_s^t\frac{1}{1+n^2(t-r)} \frac{1}{1+\log(1+n^2(t-r))} (n\mf B_r^{k-1}+B_r^{k-2})dr.
 	\end{split}
 \end{equation}

 We now use \eqref{Bt} and \eqref{mfBt} to estimate $B_t^k$ and $\mf B_t^k$.
 
 For $k=1$, by Theorem \ref{onetime2d}, recalling that $A_s^1=0$ and $A_s^0=1$, we have
 \begin{equation*}
 	B_t^1\,\lesssim\,\frac{\log n}{n^2}\,+\,\frac{1}{1+n^2(t-s)}
 \end{equation*}
 and 
 \begin{equation*}
 	\mf B_t^1\,\lesssim\,\frac{1}{\sqrt{1+n^2(t-s)}}\frac{\log^2 n}{n^2}\,+\,\Big(\frac{1}{1+n^2(t-s)}\Big)^{\frac{3}{2}}.
 \end{equation*}

 For $k=2$, the sum in the first line at the right hand side of \eqref{Bt} is bounded by
 \begin{equation*}
 	\frac{C}{n^3}\log n\log\log n\,+\,\frac{C}{1+n^2(t-s)}\frac{1}{n^2}\log n.
 \end{equation*}
The integral term is bounded  by 
 \begin{equation*}
 	\int_s^t \frac{1}{1+n^2(t-r)} n\mf B_r^1dr\,\lesssim\,\frac{\log n}{n^3}\,+\,\frac{1}{n}\frac{1}{1+n^2(t-s)}
 \end{equation*}
 by  \eqref{eq1d3}.
 These estimates above give the desired bound of $B_t^2$.
  The sum in the first line at the right hand side of \eqref{mfBt} is bounded by
 \begin{equation*}
 	 \frac{(\log^2 n) \log\log n }{n^3\sqrt{1+n^2(t-s)}}\,+\, \Big(\frac{1}{1+n^2(t-s)}\Big)^{\frac{3}{2}}\frac{\log n}{n^2}.
 \end{equation*}
 The integral term is bounded  by 
 \begin{equation*}
 	\int_s^t \frac{1}{1+n^2(t-r)}\frac{1}{1+\log(1+n^2(t-r))} n\mf B_r^1dr
 \end{equation*}
 which, by  \eqref{trlogrs}, is less than or equal to 
 \begin{equation*}
 \frac{C}{\sqrt{1+n^2(t-s)}}\frac{(\log^2 n)\,\log\log n}{n^3}\,+\,\frac{C}{1+n^2(t-s)}\frac{1}{n(1+\log(1+n^2(t-s)))}.
 \end{equation*}
 Comparing the four terms from the bounds for the initial term and the integral term, we find the dominating term and get 
 \begin{equation*}
 	\mf B_t^2\,\lesssim\, \frac{1}{1+n^2(t-s)}\frac{1}{n(1+\log(1+n^2(t-s)))}.
 \end{equation*}

 For $k=3$, the sum in the first line at the right hand side of \eqref{Bt} is bounded  by
 \begin{equation*}
 	C\frac{\log^2 n}{n^4}\,+\,C\frac{\log n}{n^2}\frac{1}{1+n^2(t-s)}.
 \end{equation*}
 The integral term is bounded  by 
 \begin{equation*}
 	C\int_s^t \frac{1}{1+n^2(t-r)} (n\mf B_r^2+B_r^1)dr
 \end{equation*}
 which, by  \eqref{trlogrs}, \eqref{eqd2} and \eqref{eq1d3}, is less than or equal to 
 \begin{equation*}
 	\begin{split}
 		&\frac{1}{1+n^2(t-s)}\frac{\log\log n}{n^2}\,+\,\frac{(\log n)\log(1+n^2(t-s))}{n^4}\,+\,\frac{\log(1+n^2(t-s))}{n^2(1+n^2(t-s))}\\
 		\lesssim\,&\frac{\log^2 n}{n^4}\,+\,\frac{1}{1+n^2(t-s)}\frac{\log n}{n^2}.
 	\end{split}
 \end{equation*}
 Combining all the estimates above, we finish the proof.
\end{proof}

The following lemma was used to prove \eqref{eqn 1}.
\begin{lemma}[two times with two points at lower time]\label{twotimestwopt}
	Fix $0< s<t\leq T$ and two distinct points $z_1,z_2\in\bb Z^d$. Assume $x_1,x_2$ are distinct. Then 
	\begin{equation*}
		\begin{split}
			|\phi(s,s,t;z_1,z_2,x_1)|\,	\lesssim\,
			\begin{cases}
			\frac{(\log n)\log\log n}{n^3}\,+\,\frac{\log n}{n^2}\frac{1}{n^2(t-s)+1} \quad &\text{if}\;\; d=2\\
				\frac{1}{n^3}\,+\, \frac{1}{n^2}\Big(\frac{1}{n^2(t-s)+1}\Big)^{\frac{d}{2}} \quad &\text{if}\;\; d\geq 3,
			\end{cases}
		\end{split}
	\end{equation*}
	and
	\begin{equation*}
		\begin{split}
			|\phi(s,s,t,t;z_1,z_2,x_1,x_2)|\,\lesssim\,
			\begin{cases}
				\frac{(\log^2 n)\log\log n}{n^4}\,+\,\frac{\log n}{n^2}\frac{1}{n^2(t-s)+1}\,+\,\Big(\frac{1}{n^2(t-s)+1}\Big)^{2} \quad &\text{if}\;\; d=2\\
					\frac{1}{n^4}\,+\,\frac{1}{n^2}\Big(\frac{1}{n^2(t-s)+1}\Big)^{\frac{d}{2}}\,+\,\Big(\frac{1}{n^2(t-s)+1}\Big)^{d} \quad &\text{if}\;\; d\geq 3.
			\end{cases}
		\end{split}
	\end{equation*}
\end{lemma}
\begin{proof}
	Denote 
\begin{equation*}
	Q_t^k\,=\,\sup_{x_i's\; \text{are distinct}} 	|\phi(s,s,t,\cdots,t;z_1,z_2,x_1,\cdots,x_k)|.
\end{equation*}
Then obviously $Q_t^0=A_s^2$. 

For every $d\geq 2$, we can obtain an inequality similar to \eqref{ind2}: 
\begin{equation}\label{ind3t2}
	\begin{split}
		&|\phi(s,s,t,\cdots,t;z_1,z_2,x_1,\cdots,x_k)|\\
		\lesssim\, & \sum_{(y_i:1\leq i\leq k)\in\Lambda_{k}}p_{t-s}((x_i:1\leq i\leq k),(y_i:1\leq i\leq k))\,	|\phi(s,\cdots,s;z_1,z_2,y_1,\cdots,y_k)|\\ +\,&\one_{\{k\geq 2\}}\int_s^t\Big(\frac{1}{1+n^2(t-r)}\Big)^{\frac{d}{2}} (nQ_r^{k-1}+Q_r^{k-2})dr.
	\end{split}
\end{equation}
The initial term now should be estimated by considering whether some $y_i$  is equal to $z_1,z_2$ or not. For points $\mathbf y$ such that $y_i\neq z_1,z_2$ for all $i$,  the total contribution is bounded simply by $A_s^{k+2}$. If there exists one and only one $i$ such that $y_i$ equal to one of $z_1,z_2$, then by Lemma \ref{fixy} below with $|I|=1$ and \eqref{reduce}, the total contribution is bounded by
\begin{equation*}
	C\Big(\frac{1}{n^2(t-s)+1}\Big)^{d/2} (A_s^{k+1}+A_s^{k}).
\end{equation*}
For points such that there are two coordinates equal to $z_1,z_2$ respectively,  we use Lemma \ref{fixy} below with $|I|=2$ and \eqref{reduce1} to estimate the corresponding total contribution by
\begin{equation*}
		C\Big(\frac{1}{n^2(t-s)+1}\Big)^{d} (A_s^{k}+A_s^{k-1}+A_s^{k-2}).
	\end{equation*}
Combining these estimates together, we have 
\begin{equation}\label{indtwoQ}
	\begin{split}
		Q_t^k\lesssim &\, A_s^{k+2}\,+\,\Big(\frac{1}{n^2(t-s)+1}\Big)^{\frac{d}{2}} (A_s^{k+1}+A_s^{k})\\
		+\,& \one_{\{k\geq 2\}}\Big(\frac{1}{n^2(t-s)+1}\Big)^{d} (A_s^{k-1}+A_s^{k-2})\\ +\,&\one_{\{k\geq 2\}}\int_s^t\Big(\frac{1}{1+n^2(t-r)}\Big)^{\frac{d}{2}} (nQ_r^{k-1}+Q_r^{k-2})dr.
	\end{split}
\end{equation}

We now use this recursion relation to estimate $Q_t^1$ and $Q_t^2$. From \eqref{indtwoQ},
\begin{equation*}
	\begin{split}
		Q_t^1\,\lesssim\,&A_s^3\,+\,\Big(\frac{1}{n^2(t-s)+1}\Big)^{\frac{d}{2}} (A_s^2+A_s^1)
	\end{split}
\end{equation*}
which can be further bounded by
\begin{equation*}
	\frac{(\log n)\log\log n}{n^3}\,+\,\frac{1}{n^2(t-s)+1} \frac{\log n}{n^2}
\end{equation*}
if $d=2$
and
\begin{equation*}
\frac{1}{n^3}\,+\,\Big(\frac{1}{n^2(t-s)+1}\Big)^{\frac{d}{2}} \frac{1}{n^2}
\end{equation*}
if $d\geq 3$. 
 Moreover, recalling that $A_s^1=0$ and $A_s^0=1$, we have
\begin{equation*}
	\begin{split}
		Q_t^2\lesssim &\, A_s^4\,+\,\Big(\frac{1}{n^2(t-s)+1}\Big)^{\frac{d}{2}} (A_s^3+A_s^2)\,+\,\Big(\frac{1}{n^2(t-s)+1}\Big)^{d} \\ +\,&\int_s^t\Big(\frac{1}{1+n^2(t-r)}\Big)^{\frac{d}{2}} (nQ_r^1+A_s^2)dr,
	\end{split}
\end{equation*}
which is less than or equal to 
\begin{equation*}
	\frac{(\log^2 n)\log\log n}{n^4}\,+\,\frac{\log n}{n^2}\frac{1}{n^2(t-s)+1}\,+\,\Big(\frac{1}{n^2(t-s)+1}\Big)^{2}, \quad \text{if}\,\, d=2
\end{equation*}
using \eqref{eqd2} and \eqref{eq1d3}
and
\begin{equation*}
	\frac{1}{n^4}\,+\,\frac{1}{n^2}\Big(\frac{1}{n^2(t-s)+1}\Big)^{\frac{d}{2}}\,+\,\Big(\frac{1}{n^2(t-s)+1}\Big)^{d}, \quad \text{if}\,\, d\geq 3
\end{equation*}
 using \eqref{eqd3} and \eqref{eq3d3}.
\end{proof}
\begin{remark}
The bounds in the lemma for $d=2$ can be slightly improved by applying the strategy used in Lemma \ref{twotimes} for the two-dimensional estimates. However, since these bounds already suffice to verify \eqref{convergence 1}, we did not pursue sharper estimates.
\end{remark}

\subsection{Estimates on the transition probability of stirring process}
In this subsection, we provide estimates used in previous two subsections. Recall that $p_t$ is the transition probability of the (accelerated) stirring process with labeled particles.

\begin{lemma}\label{fixy}
	Fix $\mathbf x\in\Lambda_k$ and a non-empty subset $I\subset \{1,\cdots,k\}$ and a list of  distinct points $\{z_j,j\in I\}$ in $\bb Z^d$ with $d\geq 1$. Then there exists a constant $C$ independent of $\mathbf x$, $z_j$'s and $I$ such that 
	\begin{equation*}
		\sum_{\substack{\mathbf y\in\Lambda_{k}\\y_j=z_j,\forall j\in I}} p_t(\mathbf x,\mathbf y)\,\leq\, C\Big(\frac{1}{1+n^2t}\Big)^{\frac{d}{2}|I|}.
	\end{equation*}
\end{lemma}
\begin{proof}
	It was proved in Lemma 3.2 and Lemma 3.5 of \cite{erhard2024scaling} that 
	\begin{equation*}
		p_t(\mathbf x,\mathbf y)\,\lesssim\,\prod_{i=1}^k \bar p_t(x_i,y_i)
	\end{equation*}
	where $\bar p_t(x,y)$ satisfies 
	\begin{equation}\label{bar1}
		\bar p_t(x,y)\,\lesssim\,\Big(\frac{1}{1+n^2t+|x-y|^2}\Big)^{\frac{d}{2}},
	\end{equation}
	\begin{equation}\label{bar2}
		\sum_{y\in\bb Z^d} \bar p_t(x,y)\,\lesssim\,1\quad \text{and}\quad 	\sum_{x\in\bb Z^d} \bar p_t(x,y)\,\lesssim\,1.
	\end{equation}
	We emphasize that in the above formulas,  $a\lesssim b$ means that there exists a constant $C$ independent of $\mathbf x,\mathbf y,x,y$ and $t$ such that $a\leq Cb$. As a consequence,  
	\begin{equation*}
	\sum_{\substack{\mathbf y\in\Lambda_{k}\\y_j=z_j,\forall j\in I}} p_t(\mathbf x,\mathbf y) \,\lesssim\, \prod_{j\in I}\bar p_t(x_j,z_j)\,\lesssim\,\Big(\frac{1}{1+n^2t}\Big)^{\frac{d}{2}|I|}.
	\end{equation*}
\end{proof}

The next lemma was used in the proof of Lemma \ref{onetime}. Recall that $\mathbf X_t\in\Lambda_{k}$ is the vector of positions of labeled particles of stirring process.  
\begin{lemma}\label{coupling}
	Assume $d\geq 1$. Then for any $t>0$,  
	\begin{equation*}
		\sum_{i,j=1}^{k} \mathbf{P}_{(x_i\,:\, 1\leq i\leq k)}[|\mathbf{X}_{t}^j- \mathbf{X}_{t}^i|=1]\,\lesssim\,\Big(\frac{1}{1+n^2t}\Big)^{\frac{d}{2}}.
	\end{equation*}
\end{lemma}

\begin{proof}
	Since $k$ is finite and each site can hold at most one stirring particle, it is enough to prove that
	\begin{equation}\label{Pmin}
		\mathbf{P}_{(x_i\,:\, 1\leq i\leq k)}\big[|\mathbf{X}_{t}^j- \mathbf{X}_{t}^i|=1\big]\,\lesssim\,\Big(\frac{1}{1+n^2t}\Big)^{\frac{d}{2}},
	\end{equation}
	for any $i,j$. Without loss of generality, we shall assume $i=1,j=2$.
	
	In view of Lemma 3.2 and Lemma 3.5 of \cite{erhard2024scaling},
	\begin{equation*}
		\begin{split}
		\mathbf{P}_{(x_i\,:\, 1\leq i\leq k)}\big[|\mathbf{X}_{t}^1- \mathbf{X}_{t}^2|=1\big]&\,=\,\sum_{\substack{\mathbf y\in\Lambda_k\\|y_1-y_2|=1}}p_t(\mathbf x,\mathbf y)\,\lesssim\,\sum_{\substack{\mathbf y\in\Lambda_k\\|y_1-y_2|=1}}\prod_{i=1}^k \bar p_t(x_i,y_i)\\
		\leq\, &\sum_{\substack{\mathbf y\in\Lambda_k\\|y_1-y_2|=1}}\sup_{z}\bar p_t(x_1,z)\prod_{i=2}^k \bar p_t(x_i,y_i)\,\lesssim\Big(\frac{1}{1+n^2t}\Big)^{\frac{d}{2}},
		\end{split}
	\end{equation*}
	where in the last inequality we used \eqref{bar1} and \eqref{bar2} specifically.

\end{proof}

\begin{lemma}\label{difp}
There exists a constant universal $C=C(k)$ such that 
\begin{equation*}
	\sum_{\mathbf y\in\Lambda_{k}}\big\lvert p_t(\mathbf x,\mathbf y)\,-\,p_t(\mathbf x+\mathbf e_{11},\mathbf y) \big\rvert\,\leq\,	\begin{cases}
		\frac{C\log(1+n^2t)}{\sqrt{1+n^2t}} \quad &\text{if}\;\; d=2\\
		\frac{C}{\sqrt{1+n^2t}} \quad &\text{if}\;\; d\geq 3,
	\end{cases}
\end{equation*}
valid for all $\mathbf x, \mathbf x+\mathbf e_{11}\in\Lambda_k$ and $t\geq 0$. In the above formula,
\begin{equation*}
	\mathbf e_{11}\,=\,((1,0,\cdots,0),(0,\cdots,0),\cdots,(0,\cdots,0))\in (\bb Z^d)^k.
\end{equation*}
\end{lemma}

In fact, when proving equation (3.2) of  \cite{erhard2025scaling}, the statement of Lemma \ref{difp} is what actually was proved by using Proposition 3.6 and Lemma 3.8 there. So we omit the proof here.

\begin{lemma}\label{difpfixz}
	Fix $T>0$, $z\in\bb Z^d$ and $1\leq i\leq k$. There exists a universal  constant $C=C(k)$ independent of $z$ and $i$ such that, for all $t\leq T$ and $\mathbf x, \mathbf x+\mathbf e_{11}\in\Lambda_k$,
	\begin{equation*}
		\sum_{\mathbf y\in\Lambda_{k},y_i=z}\big\lvert p_t(\mathbf x,\mathbf y)\,-\,p_t(\mathbf x+\mathbf e_{11},\mathbf y) \big\rvert\,\leq\,	C\Big(\frac{1}{1+n^2t}\Big)^{\frac{d+1}{2}}.
	\end{equation*}
\end{lemma}
\begin{proof}
	Write $p^{rw}_t$ for the transition probability of $k$ independent (accelerated by $n^2$) random walks evolving on $\bb Z^d$ and recall that $q_t$ is the transition probability of the (accelerated) random walk on $\bb Z^d$. Then
	\begin{equation*}
		p_t^{rw}(\mathbf x,\mathbf y)\,=\,\prod_{i=1}^{k} q_t(x_i,y_i).
	\end{equation*}
	For $i\neq 1$, using Proposition 3.6 of \cite{erhard2025scaling} for $k-1$ walks and Lemma \ref{fixy} for one particle, we have 
	\begin{equation}\label{rwbound}
		\sum_{\mathbf y\in (\bb Z^d)^k,y_i=z}\big\lvert p^{rw}_t(\mathbf x,\mathbf y)\,-\,p^{rw}_t(\mathbf x+\mathbf e_{11},\mathbf y) \big\rvert\,\lesssim\, \frac{1}{\sqrt{1+n^2t}}q_t(x_i,z)\,\lesssim\, \Big(\frac{1}{1+n^2t}\Big)^{\frac{d+1}{2}}.
	\end{equation}
	For $i=1$, we can use equation (3.4) of \cite{erhard2025scaling} to get the same bound as in \eqref{rwbound}.
	
	We now compare $p_t$ and $p_t^{rw}$. We claim that, there exists a constant $C=C(k)$ independent of $z$ and $i$ such that for any $\mathbf x\in\Lambda_k$,
	\begin{equation}\label{rwcomp}
		\sum_{\mathbf y\in\Lambda_{k},y_i=z}|p_t^{rw}(\mathbf x,\mathbf y)-p_t(\mathbf x,\mathbf y)|\,\leq\, \frac{C}{n^2}\Big(\frac{1}{1+n^2 t} \Big)^{\frac{d}{2}},\quad \text{if}\;\;d\geq 2.
	\end{equation}
	Indeed, as a simple consequence of \eqref{rwbound} and the triangle inequality, there exists a universal constant $C$ independent of $\mathbf w$ such that for every $\mathbf w$ with $|w_j-w_\ell|=1$ for some $j<\ell$,
	\begin{equation*}
		\sum_{\mathbf y\in(\bb Z^d)^k,y_i=z}\big\lvert p^{rw}_t (\bar{\mathbf w},\mathbf y)\,-\,p^{rw}_t (\mathbf w+\mathbf e_{11},\mathbf y) \big\rvert\,\lesssim\, \Big(\frac{1}{1+n^2t}\Big)^{\frac{d+1}{2}}, 
	\end{equation*}
	for every $\bar{\mathbf w}\in\{\delta^{j,\ell}\mathbf w, \delta^{\ell,j}\mathbf w, \sigma^{j,\ell}\mathbf w\}$, where $\delta^{i,j}\mathbf w\in(\bb Z^d)^k$ is defined by 
	\begin{equation*}
		(\delta^{i,j}\mathbf w)_\ell\,=\,\mathbf w_\ell, \,\forall\;\ell\neq j,\quad \text{and}\quad (\delta^{i,j}\mathbf w)_j\,=\,\mathbf w_i,
	\end{equation*}
	and $\sigma^{i,j}\mathbf w\in(\bb Z^d)^k$ is defined by 
	\begin{equation*}
		(\sigma^{i,j}\mathbf w)_\ell\,=\,\mathbf w_\ell, \,\forall\;\ell\neq i,j,\quad (\sigma^{i,j}\mathbf w)_i\,=\,\mathbf w_j,\quad \text{and}\quad (\sigma^{i,j}\mathbf w)_j\,=\,\mathbf w_i.
	\end{equation*}
	This bound together with  lemma 3.7 and equation (3.10) of \cite{erhard2025scaling} yields that, the sum to be estimated in the claim is less than or equal to 
	\begin{equation*}
		C\int_0^t \Big(\frac{1}{1+n^2(t-s)}\Big)^{\frac{d}{2}}\Big(\frac{1}{1+n^2s}\Big)^{\frac{d+1}{2}}ds,
	\end{equation*}
	which gives the desired bound in the claim after an elementary computation. 
	
	The lemma follows immediately from \eqref{rwcomp} and \eqref{rwbound}. 
\end{proof}

	\subsection{Correlation estimates at three times}
	\begin{lemma}[three times]\label{lem three times}
		Fix $0< s_1<s<t\leq T$ and $y\in\bb Z^d$. Assume  $x_1,x_2$ are distinct. Then 
			\begin{equation*}
			\begin{split}
				&\phi(s_1,s,t;y,y,x_1)\\
				\lesssim&
				\begin{cases}
					\frac{\log n\log\log n}{n^3}\,+\, \frac{1}{n(1+n^2(s-s_1))}\,+\,\frac{1}{n^2(t-s)+1}\Big\{\frac{\log n}{n^2}\,+\,\frac{1}{1+n^2(s-s_1)}\Big\} \quad &\text{if}\;\; d=2\\
						\frac{1}{n^3}\,+\,\frac{1}{n}\Big(\frac{1}{1+n^2(s-s_1)}\Big)^{\frac{d}{2}}\,+\,\Big(\frac{1}{n^2(t-s)+1}\Big)^{\frac{d}{2}} \Big\{\frac{1}{n^2}\,+\,\Big(\frac{1}{1+n^2(s-s_1)}\Big)^{\frac{d}{2}}\Big\}  \quad &\text{if}\;\; d\geq 3.
				\end{cases}
			\end{split}
		\end{equation*}
		and
		\begin{equation*}
			\begin{split}
				&\phi(s_1,s,t,t;y,y,x_1,x_2)\\
				\lesssim\,&
				\begin{cases}
					 \frac{\log^2 n}{n^4}\,+\,\frac{\log n}{n^2(1+n^2(s-s_1))}\,+\,\frac{1}{n^2(t-s)+1}\Big(\frac{\log n}{n^2}\,+\,\frac{1}{1+n^2(s-s_1)}\Big) \quad &\text{if}\;\; d=2\\
					 \frac{1}{n^4}\,+\,\frac{1}{n^2}\Big(\frac{1}{1+n^2(s-s_1)}\Big)^{\frac{d}{2}} \,+\,\Big(\frac{1}{1+n^2(t-s)}\Big)^{d/2}\Big(	\frac{1}{n^2}\,+\,\Big(\frac{1}{1+n^2(s-s_1)}\Big)^{\frac{d}{2}} \Big) \quad &\text{if}\;\; d\geq 3.
				\end{cases}
			\end{split}
		\end{equation*}
	\end{lemma}
	\begin{proof}
			Let us denote 
		\begin{equation*}
			D_t^k\,=\,\sup_{x_i's\; \text{are distinct}} 	|\phi(s_1, s,t,\cdots,t;y,y,x_1,\cdots,x_k)|.
		\end{equation*}
	Then obviously $D_t^0=B_{s}^1$. A similar argument to that used for \eqref{indtwo} yields the recursion: for $D_t^k$:
		\begin{equation}\label{indthree}
			\begin{split}
				D_t^k\lesssim &\, B_{s}^{k+1}\,+\,\Big(\frac{1}{n^2(t-s)+1}\Big)^{\frac{d}{2}} (B_{s}^k+B_{s}^{k-1})\\ 
				+\,&\one_{\{k\geq 2\}}\int_{s}^t\Big(\frac{1}{1+n^2(t-r)}\Big)^{\frac{d}{2}} (nD_r^{k-1}+D_r^{k-2})dr,
			\end{split}
		\end{equation}
		for all $k\geq 1$. We use this recursion relation to estimate $D_t^k$ for $d\geq 3$.

		For $k=1$, $D_t^1$ can be bounded by 
		\begin{equation*}
			\frac{C}{n^3}\,+\,\frac{C}{n}\Big(\frac{1}{1+n^2(s-s_1)}\Big)^{\frac{d}{2}}\,+\,C\Big(\frac{1}{n^2(t-s)+1}\Big)^{\frac{d}{2}} \Big\{\frac{1}{n^2}\,+\,\Big(\frac{1}{1+n^2(s-s_1)}\Big)^{\frac{d}{2}}\Big\}.
		\end{equation*}

		For $k=2$, the sum of the first two terms at the right hand side of \eqref{indthree} is bounded by
		\begin{equation*}
			 \frac{C}{n^4}\,+\,\frac{C}{n^2}\Big(\frac{1}{1+n^2(s-s_1)}\Big)^{\frac{d}{2}} \,+\,C\Big(\frac{1}{1+n^2(t-s)}\Big)^{d/2}\Big(	\frac{1}{n^2}\,+\,\Big(\frac{1}{1+n^2(s-s_1)}\Big)^{\frac{d}{2}} \Big).
		\end{equation*}
		The integral term becomes 
		\begin{equation*}
			\int_s^t \Big(\frac{1}{1+n^2(t-r)}\Big)^{\frac{d}{2}} (nD_r^1+B_s^1)dr.
		\end{equation*}
Using the bound for $D_r^1$ and $B_s^1$, this is bounded by
	 \begin{equation*}
	 	\begin{split}
	 		\int_s^t\Big(\frac{1}{1+n^2(t-r)}\Big)^{\frac{d}{2}}\Big\{&\frac{1}{n^2}\,+\,\Big(\frac{1}{1+n^2(s-s_1)}\Big)^{\frac{d}{2}}\\
	 		+\,&\Big(\frac{1}{n^2(r-s)+1}\Big)^{\frac{d}{2}}\Big\{\frac{1}{n}\,+\,n\Big(\frac{1}{1+n^2(s-s_1)}\Big)^{\frac{d}{2}}\Big\} \\
	 		+\,&\frac{1}{n^2}\,+\,\Big(\frac{1}{1+n^2(s-s_1)}\Big)^{\frac{d}{2}}\Big\}dr.
	 	\end{split}
	 \end{equation*}
By \eqref{eqd3} and \eqref{eq3d3}, this integral is bounded by: 
		\begin{equation*}
				\frac{C}{n^4}\,+\,\frac{C}{n^2}\Big(\frac{1}{1+n^2(s-s_1)}\Big)^{\frac{d}{2}}\,+\,\frac{C}{n^2}\Big(\frac{1}{1+n^2(t-s)}\Big)^{\frac{d}{2}}\Big\{\frac{1}{n}\,+\,n\Big(\frac{1}{1+n^2(s-s_1)}\Big)^{\frac{d}{2}}\Big\}.
		\end{equation*} 
	 Combining the above terms gives the desired bound for $D_t^2$.
	 
	  We now deal with the case $d=2$. Define 
	 \begin{equation*}
	 	\mf D_t^{k-1}\,=\,\sup_{\substack{\mathbf x\in \Lambda_{k}\\1\leq i,j\leq k}}	\Big|\one\{|x_j-x_i|=1\}\big[\phi(s_1,s,t,\cdots, t;y, y,\mathbf x\backslash x_j)-\phi(s_1,s,t,\cdots,t;y,y,\mathbf x\backslash x_i)\big]\Big|
	 \end{equation*}
	 for every $k\geq 1$. Similar to \eqref{Bt} and \eqref{mfBt}, we obtain the recursions:
	 \begin{equation}\label{Dt}
	 	\begin{split}
	 		D_t^k\,\lesssim\, &B_s^{k+1}\,+\,\frac{1}{n^2(t-s)+1} (B_s^k+B_s^{k-1})\\
	 		+\,&\one_{\{k\geq 2\}}\int_s^t \frac{1}{1+n^2(t-r)} (n\mf D_r^{k-1}+D_r^{k-2})dr.
	 	\end{split}
	 \end{equation}
	 and
	  \begin{equation}\label{mfDt}
	 	\begin{split}
	 		\mf D_t^k\,\lesssim\,&   \frac{\log(1+n^2(t-s))}{\sqrt{1+n^2(t-s)}}	B_s^{k+1}\,+\, \Big(\frac{1}{1+n^2(t-s)}\Big)^{\frac{3}{2}}\big[B_s^k+B_s^{k-1}\big]\\
	 		+\,& \one_{\{k\geq 2\}}\int_s^t\frac{1}{1+n^2(t-r)} \frac{1}{1+\log(1+n^2(t-r))} (n\mf D_r^{k-1}+D_r^{k-2})dr.
	 	\end{split}
	 \end{equation}

	 	For $k=1$, by \eqref{Dt}, $D_t^1$ can be bounded by a constant multiple of
  \begin{equation*}
 		\frac{\log n\log\log n}{n^3}\,+\,\frac{1}{n(1+n^2(s-s_1))}\,+\,\frac{1}{n^2(t-s)+1}\Big\{\frac{\log n}{n^2}\,+\,\frac{1}{1+n^2(s-s_1)}\Big\}.
 \end{equation*}
	 By \eqref{mfDt}, $\mf D_t^1$ is bounded by a constant multiple of
	 \begin{equation*}
	 	\frac{\log(1+n^2(t-s))}{\sqrt{1+n^2(t-s)}}	\big[\frac{\log n\log\log n}{n^3}+\frac{1}{n(1+n^2(s-s_1))}\big]\,+\, \Big(\frac{1}{1+n^2(t-s)}\Big)^{\frac{3}{2}}\big[\frac{\log n}{n^2}\,+\,\frac{1}{1+n^2(s-s_1)}\big].
	 \end{equation*}

	 For $k=2$, the sum of the first two terms at the right hand side of \eqref{Dt} is bounded by
	 \begin{equation*}
	 	\frac{C\log^2 n}{n^4}\,+\,\frac{C\log n}{n^2(1+n^2(s-s_1))}\,+\,\frac{C}{n^2(t-s)+1}\Big(\frac{\log n}{n^2}\,+\,\frac{1}{1+n^2(s-s_1)}\Big).
	 \end{equation*}
	 The integral term becomes 
	 \begin{equation*}
	 	\int_s^t \frac{1}{1+n^2(t-r)} (n\mf D_r^1+B_s^1)dr.
	 \end{equation*}
	 Using bounds for $\mf D_r^1$ and $B_s^1$, this can be bounded by
	 \begin{equation*}
	 	\begin{split}
	 		\int_s^t&\frac{1}{1+n^2(t-r)}\Big\{\frac{\log(1+n^2(t-s))}{\sqrt{1+n^2(r-s)}}\big[\frac{\log n\log\log n}{n^2}+\frac{1}{1+n^2(s-s_1)}\big]\\
	 		&\quad+\,\Big(\frac{1}{1+n^2(r-s)}\Big)^{\frac{3}{2}}\big[\frac{\log n}{n}\,+\,\frac{n}{1+n^2(s-s_1)}\big]\Big\} \\
	 		&\quad\quad\quad+\,\frac{1}{1+n^2(t-r)}\Big\{\frac{\log n}{n^2}\,+\,\frac{1}{1+n^2(s-s_1)}\Big\}dr
	 	\end{split}
	 \end{equation*}
By \eqref{eq1d3} and \eqref{eqd2}, this integral is bounded by a constant multiple of
	 \begin{equation*}
	 	\begin{split}
	 		&\frac{\log n \log\log n}{n^4}\,+\,\frac{1}{n^2}\frac{1}{1+n^2(s-s_1)}\,+\, \frac{1}{1+n^2(t-s)}\big[\frac{\log n}{n^3}\,+\,\frac{1}{n(1+n^2(s-s_1))}\big]\\
	 		+\,&\frac{1}{n^2} \log(1+n^2 (t-s))\Big\{\frac{\log n}{n^2}\,+\,\frac{1}{1+n^2(s-s_1)}\Big\}.
	 	\end{split}
	 \end{equation*}
	\ We can further bound it by
	 \begin{equation*}
	 	\frac{\log^2n}{n^4}\,+\,\frac{\log n }{n^3(1+n^2(t-s))}\,+\,\frac{1}{n(1+n^2(t-s))(1+n^2(s-s_1))}\,+\,\frac{\log n}{n^2}\frac{1}{1+n^2(s-s_1)}.
	 \end{equation*}
	Obviously the bound from the initial term dominates. We finish the proof.
			\end{proof}
		
	\subsection{Correlation estimates at four times}
	
		\begin{lemma}[four times]\label{lem four times}
			Fix $0< s_1<s_2<s<t\leq T$ and $y\in\bb Z^d$.  Then for $d=2$,
			\begin{equation*}
				\begin{split}
					&\phi(s_1,s_2,s,t;y,y,y,x)\\
					\lesssim\,&\frac{\log^2 n}{n^4}\,+\,\frac{\log n}{n^2(1+n^2(s_2-s_1))}\,+\,\frac{1}{1+n^2(s-s_2)}\Big(\frac{\log n}{n^2}\,+\,\frac{1}{1+n^2(s_2-s_1)}\Big)\\
					+\,&\frac{1}{1+n^2(t-s)}\Big(\frac{\log n}{n^2}\,+\,\frac{1}{1+n^2(s_2-s_1)}\Big);
				\end{split}
			\end{equation*}
			for $d\geq 3$,
			\begin{equation*}
				\begin{split}
					&\frac{1}{n^4}\,+\,\frac{1}{n^2}\Big(\frac{1}{1+n^2(s_2-s_1)}\Big)^{\frac{d}{2}} \,+\,\Big(\frac{1}{1+n^2(s-s_2)}\Big)^{d/2}\Big(\frac{1}{n^2}\,+\,\Big(\frac{1}{1+n^2(s_2-s_1)}\Big)^{\frac{d}{2}} \Big)\\
					+\,&\Big(\frac{1}{1+n^2(t-s)}\Big)^{\frac{d}{2}}\Big(\frac{1}{n^2}\,+\,\Big(\frac{1}{1+n^2(s_2-s_1)}\Big)^{\frac{d}{2}}\Big),
				\end{split}
			\end{equation*}
			valid for all $x\in\bb Z^d$.
		\end{lemma}
		\begin{proof}
			Repeating the previous procedure, we can get
			\begin{equation}\label{sumD}
				|\phi(s_1,s_2,s,t;y,y,y,x)|\,\lesssim\,D_{s}^2\,+\,\Big(\frac{1}{n^2(t-s)+1}\Big)^{\frac{d}{2}} (D_{s}^0+D_s^1).
			\end{equation}
			Note that $D_s^0=B_{s_2}^1$. Moreover $D_s^0+D_s^1$ can be bounded by
			\begin{equation*}
				\frac{\log n}{n^2}\,+\,\frac{1}{1+n^2(s_2-s_1)}
			\end{equation*}
			if $d=2$, and 
			\begin{equation*}
				\frac{1}{n^2}\,+\,\Big(\frac{1}{1+n^2(s_2-s_1)}\Big)^{\frac{d}{2}}
			\end{equation*}
			if $d\geq 3$.
		We conclude the proof by Lemma \ref{lem three times}.
		\end{proof}
		
	\section{Tightness}\label{sec tightness} 
	
	In this section, we prove the tightness of the sequence $\Gamma^n$ in the space $C([0,T],\R)$.  By Kolmogorov–Chentsov criterion, it suffices to prove the following result.
	
	\begin{proposition}\label{prop occupation time L4}
		There exists some constant $C$ such that for any $0 \leq s < t \leq T$,
		\[\E^n_{\nu^n_{\rho_0 (\cdot)}} \Big[\Big(\Gamma^n (t) - \Gamma^n (s)\Big)^4\Big] \leq C (t-s)^2.\]
	\end{proposition}
	\begin{proof}
		Since the proofs for the cases $d\geq 3$ and $d=2$ are similar, we only present the  proof for the latter case.
		Since $T>0$ is fixed, it is sufficient to show that 
		\begin{equation*}
			\begin{split}
				\E^n_{\nu^n_{\rho_0 (\cdot)}}  \Big[\big(\int_s^t \bar\eta_r(0)dr\big)^4\Big]
				\lesssim
				\,&\int_s^t\int_s^{\tau}\int_s^{s_3}\int_s^{s_2}  |\E^n_{\nu^n_{\rho_0 (\cdot)}} [\bar\eta_\tau(0)\bar\eta_{s_1}(0)\bar\eta_{s_2}(0)\bar\eta_{s_3}(0)]| ds_1  ds_2ds_3d\tau\\
				\lesssim\,& \frac{(t-s)^2\log^2n}{n^4}\,+\,\frac{(t-s)^3\log^2n}{n^4}\,+\,\frac{(t-s)^4\log^2 n}{n^4}.
			\end{split}
		\end{equation*}
		By Lemma \ref{lem four times}, it suffices to show that
		\begin{equation}\label{2d1}
			J_1\,=\,\int_s^t\int_s^{\tau}\int_s^{s_3}\int_s^{s_2}  \frac{\log^2 n}{n^4} ds_1  ds_2ds_3d\tau\,\lesssim\,\frac{(t-s)^4\log^2 n}{n^4}
		\end{equation}
		\begin{equation}\label{2d2}
			J_2\,=\,\int_s^t\int_s^{\tau}\int_s^{s_3}\int_s^{s_2}  \frac{1}{n^2}\frac{\log n}{(1+n^2(s_2-s_1))}ds_1  ds_2ds_3d\tau\,\lesssim\,\frac{(t-s)^3\log^2 n}{n^4}
		\end{equation}
		\begin{equation}\label{2d3}
			\begin{split}
				J_3\,=\,&\int_s^t\int_s^{\tau}\int_s^{s_3}\int_s^{s_2}  \frac{1}{1+n^2(s_3-s_2)}\Big(\frac{\log n}{n^2}\,+\,\frac{1}{1+n^2(s_2-s_1)}\Big) ds_1  ds_2ds_3d\tau\\
				\lesssim\,&\frac{(t-s)^2\log^2 n}{n^4}\,+\,\frac{(t-s)^3\log^2n}{n^4}
			\end{split}
		\end{equation}
		\begin{equation}\label{2d4}
			\begin{split}
				J_4\,=\,&\int_s^t\int_s^{\tau}\int_s^{s_3}\int_s^{s_2}  \frac{1}{1+n^2(\tau-s_3)}\Big(\frac{\log n}{n^2}\,+\,\frac{1}{1+n^2(s_2-s_1)}\Big) ds_1  ds_2ds_3d\tau\\
				\lesssim\,&\frac{(t-s)^2\log^2n}{n^4}\,+\,\frac{(t-s)^3\log^2n}{n^4}
			\end{split}
		\end{equation}
		
		\eqref{2d1} is obvious. 
		
		\eqref{2d2} follows from  \eqref{eqd2}.
		
		By \eqref{eqd2}, we have
		\begin{equation*}
			\begin{split}
				&\int_s^t\int_s^{\tau}\int_s^{s_3}\int_s^{s_2}  \frac{1}{1+n^2(s_3-s_2)}\frac{\log n}{n^2} ds_1  ds_2ds_3d\tau\\
				\lesssim\,&\frac{\log n}{n^2}(t-s)\int_s^t\int_s^{\tau}\int_s^{s_3} \frac{1}{1+n^2(s_3-s_2)}  ds_2ds_3d\tau\\
				\lesssim\,&\frac{\log^2 n}{n^4}(t-s)\int_s^t\int_s^{\tau}ds_3d\tau\,\lesssim\,\frac{\log^2 n}{n^4}(t-s)^3
			\end{split}
		\end{equation*}
		and
		\begin{equation*}
			\begin{split}
				&\int_s^t\int_s^{\tau}\int_s^{s_3}\int_s^{s_2}  \frac{1}{1+n^2(s_3-s_2)}\frac{1}{1+n^2(s_2-s_1)}ds_1  ds_2ds_3d\tau\\
				\lesssim\,&\frac{\log n}{n^2}\int_s^t\int_s^{\tau}\int_s^{s_3} \frac{1}{1+n^2(s_3-s_2)}  ds_2ds_3d\tau\\
				\lesssim\,&\frac{\log^2 n}{n^4}\int_s^t\int_s^{\tau}ds_3d\tau\,\lesssim\,\frac{\log^2 n}{n^4}(t-s)^2
			\end{split}
		\end{equation*}
		This proves \eqref{2d3}. 
		
		To see \eqref{2d4}, using \eqref{eqd2} again, we have
		\begin{equation*}
			\begin{split}
				&\int_s^t\int_s^{\tau}\int_s^{s_3}\int_s^{s_2}  \frac{1}{1+n^2(\tau-s_3)} ds_1  ds_2ds_3d\tau\\
	\lesssim\,&(t-s)^2\int_s^t\int_s^{\tau} \frac{1}{1+n^2(\tau-s_3)} ds_3d\tau\,\lesssim\,\frac{\log n}{n^2}(t-s)^3
			\end{split}
		\end{equation*}
		and 
		\begin{equation*}
			\begin{split}
				&\int_s^t\int_s^{\tau}\int_s^{s_3}\int_s^{s_2}  \frac{1}{1+n^2(\tau-s_3)}\frac{1}{1+n^2(s_2-s_1)} ds_1  ds_2ds_3d\tau\\
				\lesssim\,&\frac{\log n}{n^2}\int_s^t\int_s^{\tau}\int_s^{s_3} \frac{1}{1+n^2(\tau-s_3)} ds_2ds_3d\tau\\
				\lesssim\,&\frac{\log n}{n^2} (t-s)\int_s^t\int_s^{\tau} \frac{1}{1+n^2(\tau-s_3)} ds_3d\tau\,\lesssim\,\frac{\log^2 n}{n^4}(t-s)^2.
			\end{split}
		\end{equation*}
		\end{proof}

\appendix
	
	\section{Elementary computations}
	
	In this section, we prove some elementary integral bounds that were frequently used along the proof.
	
	\begin{lemma}\label{intsingle}
		Fix $T>0$. Then there exists a constant $C=C(T,d)$ such for for all $0\leq s<t\leq T$,
		\begin{equation}\label{eqd3}
			\int_s^t \Big(\frac{1}{1+n^2(t-r)}\Big)^{\frac{d}{2}} dr\,\leq\,\frac{C}{n^2},\quad \text{if}\;d\geq 3,
		\end{equation} 
		\begin{equation}\label{eqd2}
			\int_s^t \frac{1}{1+n^2(t-r)}  dr\,=\,\frac{1}{n^2} \log(1+n^2 (t-s)), 
		\end{equation}
		and 
		\begin{equation}\label{eqd1}
			\int_s^t \Big(\frac{1}{1+n^2(t-r)}\Big)^{\frac{1}{2}} dr\,=\,\frac{2}{n^2}\big(\sqrt{1+n^2(t-s)}-1\big). 
		\end{equation} 
	\end{lemma}
	\begin{proof}
	Let
	\begin{equation*}
		I=\int_s^t\Big(\frac{1}{1+n^2(t-r)}\Big)^{d/2}dr.
	\end{equation*}
	Change variables $u=t-r$ and then $v=n^2u$. This gives
	\begin{equation*}
		I=\int_0^{t-s}(1+n^2u)^{-d/2}du=\frac{1}{n^2}\int_0^{n^2(t-s)}(1+v)^{-d/2}dv.
	\end{equation*}
The results follow from simple calculations.
	\end{proof}
	
	\begin{lemma}
	 \begin{equation}\label{eq1d3}
		\int_s^t \frac{1}{1+n^2(t-r)}\Big(\frac{1}{1+n^2(r-s)}\Big)^{\frac{d}{2}}dr\,\lesssim
		\begin{cases}
				\frac{1}{n^2(1+\log(1+n^2(t-s)))} &\quad \text{if}\;d=1,\\
			\frac{\log(1+n^2(t-s))}{n^2(1+n^2(t-s))} &\quad \text{if}\;d=2,\\
			\frac{1}{n^2(1+n^2(t-s))} &\quad \text{if}\;d\geq 3.\\
		\end{cases}
	\end{equation}
	\end{lemma}
	\begin{proof}
		The integral to be estimated can be written as the sum $I_1+I_2$, where 
		\begin{equation*}
			I_1 = \int_s^{\frac{s+t}{2}} \frac{1}{1+n^2(t - r)} \left( \frac{1}{1+n^2(r - s)} \right)^{d/2} dr,
		\end{equation*}
		and 
		\begin{equation*}
			I_2 = \int_{\frac{s+t}{2}}^t \frac{1}{1+n^2(t - r)} \left( \frac{1}{1+n^2(r - s)} \right)^{d/2} dr.
		\end{equation*}
			Since
		\begin{equation*}
		I_1\,\lesssim\,\frac{1}{1+n^2(t-s)}\int_s^{\frac{s+t}{2}} \left( \frac{1}{1+n^2(r - s)} \right)^{d/2} dr,
		\end{equation*} 
	by Lemma \ref{intsingle}, for $d=1$,
		\begin{equation*}
			I_1 \lesssim	\frac{1}{1+n^2(t-s)}\frac{1}{n^2}\big(\sqrt{1+n^2(t-s)}-1\big)\lesssim\frac{1}{n^2(1+\log(1+n^2(t-s)))}.
		\end{equation*}
	The bounds for $I_1$ in the cases $d\geq 2$ follow from Lemma \ref{intsingle} directly.

	Similarly, by \eqref{eqd2}
		\begin{equation*}
			\begin{split}
				I_2\,\lesssim\,&\left( \frac{1}{1+n^2(t- s)} \right)^{d/2} \int_{\frac{s+t}{2}}^t \frac{1}{1+n^2(t - r)}  dr\,\\
				\lesssim\,&\left( \frac{1}{1+n^2(t- s)} \right)^{d/2}\frac{\log(1+\frac{1}{2}n^2(t-s))}{n^2},
			\end{split}
		\end{equation*}
		which can be further bounded by the upper bound that we just obtained for $I_1$.
		This conclude the proof.
	\end{proof}
	
	\begin{lemma}
		\begin{equation}\label{trlogrs}
			\begin{split}
			\int_s^t &\frac{1}{1+n^2(t-r)}\frac{1}{1+\log(1+n^2(t-r))}\Big(\frac{1}{1+n^2(r-s)}\Big)^{\frac{d}{2}}dr\\
			\lesssim\,&
			\begin{cases}
				\frac{\log\log n}{n^2} \big( \frac{1}{1+n^2(t- s)} \big)^{d/2}  &\quad \text{if}\;d=1,2,\\
			\frac{1}{n^2} 	\frac{1}{1+n^2(t-s)}\frac{1}{1+\log(1+n^2(t-s))} &\quad \text{if}\;d\geq 3.\\
			\end{cases}
			\end{split}
		\end{equation}
	\end{lemma}
	\begin{proof}
As before, we write the integral as $I_1+I_2$, where 
\begin{equation*}
	I_1 = \int_s^{\frac{s+t}{2}}\frac{1}{1+n^2(t-r)}\frac{1}{1+\log(1+n^2(t-r))}\Big(\frac{1}{1+n^2(r-s)}\Big)^{\frac{d}{2}}dr,
\end{equation*}
and 
\begin{equation*}
	I_2 = \int_{\frac{s+t}{2}}^t \frac{1}{1+n^2(t-r)}\frac{1}{1+\log(1+n^2(t-r))}\Big(\frac{1}{1+n^2(r-s)}\Big)^{\frac{d}{2}}dr.
\end{equation*}
Obviously,
\begin{equation*}
I_1\,\lesssim\, \frac{1}{1+n^2(t-s)}\frac{1}{1+\log(1+n^2(t-s))}\int_s^{\frac{s+t}{2}}\Big(\frac{1}{1+n^2(r-s)}\Big)^{\frac{d}{2}}dr.
\end{equation*}
Then, by Lemma \ref{intsingle},
\begin{itemize}
	\item for $d=1$, 
	\[	I_1\,\lesssim\,	\frac{1}{1+n^2(t-s)}\frac{1}{1+\log(1+n^2(t-s))}\frac{1}{n^2}\big(\sqrt{1+n^2(t-s)}-1\big),\]
	\item for $d=2$,
	\[I_1\,\lesssim\,	\frac{1}{1+n^2(t-s)}\frac{1}{n^2}, \]
	\item for $d\geq 3$,
	\[I_1 \,\lesssim\, 		\frac{1}{1+n^2(t-s)}\frac{1}{1+\log(1+n^2(t-s))}\frac{1}{n^2}.\]
\end{itemize}
By fundamental theorem of calculus,
\begin{equation*}
	\begin{split}
	I_2\,\lesssim\,&\left( \frac{1}{1+n^2(t- s)} \right)^{d/2} \int_{\frac{s+t}{2}}^t \frac{1}{1+n^2(t - r)} \frac{1}{1+\log(1+n^2(t-r))} dr\,\\
	\lesssim\,&\left( \frac{1}{1+n^2(t- s)} \right)^{d/2} \frac{\log(1+\log(1+n^2(t-s)))}{n^2}.
\end{split}
\end{equation*}
Combining these two bounds, we conclude the proof.
	\end{proof}

\begin{lemma}
	For any $d\geq 3$,
	\begin{equation}\label{eq3d3}
	\int_s^t \Big(\frac{1}{1+n^2(t-r)}\Big)^{\frac{d}{2}}\Big(\frac{1}{1+n^2(r-s)}\Big)^{\frac{d}{2}}dr\,\lesssim\,\frac{1}{n^2}\Big(\frac{1}{1+n^2(t-s)}\Big)^{\frac{d}{2}}.
	\end{equation}
\end{lemma}
\begin{proof}
	We write the integral to be estimated by $I_1+I_2$, where
	\begin{equation*}
		I_1 = \int_s^{\frac{s+t}{2}} \Big(\frac{1}{1+n^2(t - r)}\Big)^{\frac{d}{2}} \Big( \frac{1}{1+n^2(r - s)} \Big)^{d/2} dr,
	\end{equation*}
	and 
	\begin{equation*}
		I_2 = \int_{\frac{s+t}{2}}^t \Big(\frac{1}{1+n^2(t - r)}\Big)^{\frac{d}{2}} \Big( \frac{1}{1+n^2(r - s)} \Big)^{d/2} dr.
	\end{equation*}
	From \eqref{eqd3},
	\begin{equation*}
		I_1\,\lesssim\, \Big(\frac{1}{1+n^2(t-s)}\Big)^{\frac{d}{2}}\int_s^{\frac{s+t}{2}} \left( \frac{1}{1+n^2(r - s)} \right)^{d/2} dr\,\lesssim\,\frac{1}{n^2}\Big(\frac{1}{1+n^2(t-s)}\Big)^{\frac{d}{2}}.
	\end{equation*} 
	A similar computation shows that $I_2$ has the same upper bound.
	This conclude the proof.
\end{proof} 

{\bf \noindent Acknowledgements} Zhao thanks the financial supported by the National Natural Science Foundation of China with grant numbers 12401168 and 12371142, and the Fundamental Research Funds for the Central Universities in China. Xu thanks Funda\c c\~ao para a Ci\^encia e Tecnologia FCT/Portugal for financial support through the project ERC/FCT.

\bibliographystyle{plain}
\bibliography{bibliography}
%
%

\end{document}